\documentclass[10pt,leqno]{amsart}
\usepackage{amsmath,amssymb,latexsym}
\usepackage{amscd}
\usepackage[OT4]{fontenc}
\usepackage{lmodern}
\usepackage{mathrsfs}
\usepackage{amsthm}
\usepackage{amsfonts}
\usepackage{url}
\usepackage[fleqn,tbtags]{mathtools}
\usepackage{enumerate}
\usepackage{graphicx}
\usepackage{color}
\usepackage[all]{xy}
\usepackage{marginnote}

\theoremstyle{plain}
\newtheorem{theorem}{Theorem}[section]
\newtheorem{proposition}[theorem]{Proposition}
\newtheorem{lemma}[theorem]{Lemma}
\newtheorem{corollary}[theorem]{Corollary}

\newtheorem{problem}[theorem]{Problem}

\theoremstyle{definition}
\newtheorem{definition}[theorem]{Definition}
\newtheorem{remark}[theorem]{Remark}
\newtheorem{example}[theorem]{Example}

\newtheorem*{Acknow*}{Acknowledgements}

\theoremstyle{remark}

\newcommand{\R}{\mathbb R}
\newcommand{\N}{\mathbb N}

\newcommand{\C}{\mathbb C}
\newcommand{\D}{\mathbb D}

\newcommand{\id}{\operatorname{id}}

\newcommand{\cA}{\mathcal A}
\newcommand{\cB}{\mathscr B}
\newcommand{\cC}{\mathscr C}
\newcommand{\cD}{\mathscr D}
\newcommand{\cE}{\mathscr E}

\newcommand{\cK}{\mathscr K}
\newcommand{\cL}{\mathscr L}

\newcommand{\cP}{\mathscr P}
\newcommand{\cQ}{\mathcal Q}
\newcommand{\cM}{\mathscr M}

\newcommand{\cS}{\mathscr S}

\newcommand{\im}{\operatorname{im}}

\newcommand{\ldss}{\mathscr{L}(s',s)}
\newcommand{\lss}{\mathscr{L}(s)}
\newcommand{\ldsds}{\mathscr{L}(s')}
\newcommand{\lstars}{\mathscr{L}^*(s)}

\newcommand{\proofmainA}{{\raggedright\textbf{Proof of Theorem \ref{th_complemented_subalg_lstars}.}} \ }
\newcommand{\bqed}{\hspace*{\fill} $\Box $\medskip}

\newcommand{\proofmainB}{{\raggedright\textbf{Proof of Theorem \ref{th-complemented-bounded}.}} \ }

\let\epsilon\varepsilon
\let\phi\varphi
\let\rho\varrho

\title[Fr\'echet algebras]{Fr\'echet algebras with a dominating Hilbert algebra norm}
\author{Tomasz Cia\'s}
\date{}

\begin{document}

\begin{abstract}
Let $\mathscr{L}^*(s)$ be the maximal $\mathcal{O}^*$-algebra of unbounded operators on $\ell_2$ whose domain is the space $s$ of rapidly decreasing sequences.  
This is a noncommutative topological algebra with involution which can be identified, for instance, with the algebra $\mathscr L(s)\cap\mathscr L(s')$ or the algebra of multipliers
for the algebra $\mathscr{L}(s',s)$ of smooth compact operators.
We give a simple characterization of unital commutative Fr\'echet ${}^*$-subalgebras of $\mathscr{L}^*(s)$
isomorphic as a Fr\'echet spaces to nuclear power series spaces $\Lambda_\infty(\alpha)$ of infinite type.  
It appears that many natural Fr\'echet ${}^*$-algebras are closed ${}^*$-subalgebras of $\mathscr{L}^*(s)$, for example, the algebras $C^\infty(M)$ of smooth functions on smooth compact manifolds
and the algebra $\mathscr S (\mathbb{R}^n)$ of smooth rapidly decreasing functions on $\mathbb{R}^n$.
\end{abstract}

\maketitle

\footnotetext[1]{


{The research of the author was supported by the National Center of Science, grant no. 2013/10/A/ST1/00091}.}

\section{Introduction}
Let $s$ be the Fr\'echet space of rapidly decreasing complex sequences and let 
\[\cL^*(s):=\{x\colon s\to s: x\textrm{ is linear},s\subset\mathscr{D}(x^*)\textrm{ and }x^*(s)\subset s\},\]
where $\cD(x^*)$ is the domain of the adjoint of an unbounded operator $x$ on $\ell_2$.  
The class $\cL^*(s)$ is known as the maximal $\mathcal{O}^*$-algebra with domain $s$ and it can be seen as the largest ${}^*$-algebra of unbounded operators
on $\ell_2$ with domain $s$ -- for details see the book of Schm\"udgen \cite[Section I.2.1]{Sch}. The ${}^*$-algebra $\lstars$ can be topologised in several natural ways, as is shown in \cite[Sections I.3.3 and I.3.5]{Sch}.  
Here the space $\lstars$ is considered with -- the best from the functional analysis point of view -- locally convex topology $\tau^*$ 
(for definition see Preliminaries and also Proposition \ref{prop-top-Lstar-power-series}). 
Indeed, standard tools of functional analysis, such as closed graph theorem, open mapping theorem or uniform boundedness principle, 
can be applied to $(\lstars,\tau^*)$ (see \cite[Th. 4.5]{CiasPisz}). 
Furthermore, $\lstars$ is a topological ${}^*$-algebra -- i.e. multiplication is separately continous and involution is continous --
but it is neither locally $m$-convex nor a $\cQ$-algebra. 
The algebra $\lstars$ is isomorphic as a topological ${}^*$-algebra, for example, to the algebra $\cL(s)\cap\cL(s')$, 
the algebra of multipliers for the algebra $\ldss$ of smooth compact operators and also to the matrix algebra
\[\Lambda(\cA):=\bigg\{x=(x_{ij})\in\C^{\N^2}:\forall N\in\N\,\exists n\in\N\quad\sum_{i,j\in\N^2}|x_{ij}|\max\bigg\{\frac{i^N}{j^n},\frac{j^N}{i^n}\bigg\}<\infty\bigg\};\]
for details and more information about topological and algebraic properties of $\lstars$ we refer the reader to \cite{CiasPisz}.

The space $s$ carries all the information about nuclear Fr\'echet (even locally convex) 
spaces. Indeed, by the K\=omura-K\=omura theorem, a Fr\'echet space is nuclear if and only 
if it is isomorphic to some closed subspace of $s^\N$ (see \cite[Cor. 29.9]{MeV}). 
What about closed subspaces of the space $s$ itself? In \cite{V4} Vogt proved that a nuclear Fr\'echet space is isomorphic to a closed subspace of $s$ if and only if it has the so-called property (DN).
Moreover, quotients of $s$ were characterised by Vogt and Wagner in \cite{VW} via the so-called property ($\Omega$). Consequently, we have the following characterization: 
a nuclear Fr\'echet space is isomorphic to a complemented subspace of $s$ if and only if it has the properties (DN) and ($\Omega$). It is also well-known that a Fr\'echet space with (DN), ($\Omega$)
and a Schauder basis is isomorphic to a power series space $\Lambda_\infty(\alpha)$ of infinite type. However, it is still an open problem -- a particular case of the famous Mityagin-Pe\l{}czy\'nski problem -- 
whether there is a complemented subspace of $s$ without a basis.

In this paper, we are mainly interested in unital Fr\'echet algebras with involution which are isomorphic as Fr\'echet spaces to nuclear power series spaces of infinite type.
We show that a large class of them -- those algebras $E$ which admit a dominating Hilbert norm $||\cdot||=\sqrt{(\cdot,\cdot)}$ such that 
\begin{equation}\label{eq-xyz=yxstarz}
(xy,z)=(y,x^*z) 
\end{equation}
for all $x,y,z\in E$ -- 
can be embedded into $\lstars$ as closed, even complemented, ${}^*$-subalgebras (see Theorem \ref{th_complemented_subalg_lstars} and Remark \ref{rem-noncommutative}).
In the commutative case we even have the following characterization:
a unital commutative Fr\'echet ${}^*$-algebra isomorphic as a Fr\'echet space to a nuclear power series space $\Lambda_\infty(\alpha)$ of infinite type is isomorphic as a Fr\'echet ${}^*$-algebra to a closed ${}^*$-subalgebra
of $\lstars$ if and only if it admits a dominating Hilbert norm satisfying condition (\ref{eq-xyz=yxstarz}) (see again Theorem \ref{th_complemented_subalg_lstars}).
In Theorem \ref{th-complemented-bounded} we also characterize commutative Fr\'echet unital ${}^*$-subalgebras of $\lstars$ consisting of bounded operators on $\ell_2$ and isomorphic
as Fr\'echet space to nuclear spaces $\Lambda_\infty(\alpha)$. 
It is worth noting that condition (\ref{eq-xyz=yxstarz}) appears in the definiton of Hilbert algebras playing an important role in the theory of von Neumann algebras (see \cite[A.54]{Dix77}).

The above-mentioned results may be seen as a step towards an analogue -- in the context of nuclear power series spaces of infinite type -- 
of the celebrated commutative Gelfand-Naimark theorem. In the separable case it states that there is one to one correspondence (given by isometric ${}^*$-isomorphisms) 
between Banach algebras $C(K)$ of continuous functions on compact Hausdorff metrizable spaces $K$ and
closed unital commutative ${}^*$-subalgebras of the $C^*$-algebra $\cB(\ell_2)$ of bounded operators on $\ell_2$.

Our results are applicable. In the last section we give concrete examples of Fr\'echet ${}^*$-algebras which can be represented in $\lstars$ in the way described above.
Among them there are: the algebras $C^\infty(M)$ of smooth functions on smooth compact manifolds,
the algebras $\cE(K)$  with Schauder basis of smooth Whitney jets on compact sets $K$ with the extension property,
the algebra $\cS(\R^n)$ of smooth rapidly decreasing functions on $\R^n$,
nuclear power series algebras $\Lambda_\infty(\alpha)$ of infinite type
and the noncommutative algebra $\ldss$ of compact smooth operators.  We also provide one counterexample. We show that the unital commutative Fr\'echet ${}^*$-algebra $A^\infty(\D)$ of holomorphic 
functions on the open unit disc $\D$ with smooth boundary values is not isomorphic to any closed ${}^*$-subalgebra of $\lstars$.

\section{Preliminaries}

The canonical $\ell_2$ norm and the corresponding scalar product will be denoted by $||\cdot||_{\ell_2}$ and $\langle\cdot,\cdot\rangle$, respectively.

For locally convex spaces $E$ and $F$, we denote by $\cL(E,F)$ the space of all continuous linear operators from $E$ to $F$ and we set $\cL(E):=\cL(E,E)$.
These spaces will be considered with the topology $\tau_{\cL(E,F)}$ of uniform convergence on bounded sets.

By a \emph{topological ${}^*$-algebra} $E$ we mean a topological vector space endowed with at least separately continuous multiplication and continuous involution
which make $E$ a ${}^*$-algebra.
A \emph{Fr\'echet ${}^*$-algebra} is a topological ${}^*$-algebra whose underlying topological vector space is a Fr\'echet space 
(i.e. metrizable complete locally convex space). We do not require a Fr\'echet ${}^*$-algebra to be locally $m$-convex. 

Let $\alpha=(\alpha_j)_{j\in\N}$ be a monotonically increasing sequence in $(0,\infty)$ such that $\lim_{j\to\infty}\alpha_j=\infty$. Then
\[\Lambda_\infty(\alpha):=\{(\xi_j)_{j\in\N}\subset\C^\N\colon |\xi|_{\alpha,q}^2:=\sum_{j=1}^\infty|\xi_j|^2e^{2q\alpha_j}<\infty\quad\text{for all $q\in\N_0$}\}\]
equipped with the norms $|\cdot|_{\alpha,q}$, $q\in\N_0$, is a Fr\'echet space and it is called a \emph{power series space of infinite type}.
It appears that the space $\Lambda_\infty(\alpha)$ is nuclear if and only if $\sup_{j\in\N}\frac{\log j}{\alpha_j}<\infty$ (see e.g. \cite[Prop. 29.6]{MeV}).
In particular, for the sequence $\alpha_j:=\log j$, $j\in\N$, we obtain the space $s$ of \emph{rapidly decreasing sequences}, i.e.
\begin{equation}\label{eq_s}
s:=\{(\xi_j)_{j\in\N}\in\C^\N\colon|\xi|_q^2:=\sum_{j\in\N}|\xi_j|^2j^{2q}<\infty\quad\textrm{for all }q\in\N_0\}. 
\end{equation} 
By $s_n$ we denote the Hilbert space corresponding to the norm $|\cdot|_n$.

The strong dual of $s$ -- i.e. the space of all continuous linear functionals on $s$ with the topology of uniform convergence on bounded
subsets of $s$ (see e.g. \cite[Def. on p. 267]{MeV}) -- is isomorphic to the space   
\begin{equation}\label{eq_s'}
s':=\{(\xi_j)_{j\in\N}\in\C^\N\colon|\xi|_{-q}^2:=\sum_{j\in\N}|\xi_j|^2j^{-2q}<\infty\quad\textrm{for some }q\in\N_0\} 
\end{equation}
of \emph{slowly increasing sequences} equipped with the inductive limit topology for the sequence $(s_{-n})_{n\in\N_0}$ of Hilbert spaces corresponding to the norms $|\cdot|_{-n}$.
In other words, the locally convex topology on $s'$ is given by the family $\{|\cdot|'_{B}\}_{B\in\cB}$ of seminorms , $|\xi|'_{B}:=\sup_{\eta\in B}|\langle \eta,\xi\rangle|$,
where $\cB$ denotes the class of all bounded subsets of $s$ and, recall, $\langle\cdot,\cdot\rangle$ is the canonical scalar product on $\ell_2$. 

\begin{definition}\label{def-DN-Omega}
A Fr\'echet space $E$ with a fundamental system $(||\cdot||_q)_{q\in\N_0})$ of seminorms 
\begin{enumerate}
 \item has the \emph{property} (DN) (cf. \cite[Def. on p. 359]{MeV}) if there is a continuous norm $||\cdot||$ on $E$ -- called a \emph{dominating norm} -- such that 
for all $q\in\N_0$ there is $r\in\N_0$ and $C>0$ such that 
\[||x||_q^2\leq C||x||\;||x||_r\]
for all $x\in E$;
\item has the \emph{property} ($\Omega$) (cf. \cite[Def. on p. 367]{MeV}) if for all $p\in\N_0$ there is $q\in\N_0$ such that for all $r\in\N$ there are $\theta\in(0,1)$ and $C>0$ with
\[||y||_q^*\leq C{||y||_p^*}^{1-\theta}{||y||_r^*}^{\theta}\]
for all $y\in E'$, where $E'$ is the topological dual of $E$ and $||y||_p^*:=\sup\{|y(x)|:||x||_p\leq 1\}$. 
\end{enumerate}
\end{definition}

The properties (DN) and ($\Omega$) are linear-topological invariants which play a key role in a structure theory of nuclear Fr\'echet spaces.    
The following Theorem is due to Vogt and Wagner. 
\begin{theorem}\emph{(\cite[Ch. 31]{MeV} and \cite{VW, V4})}\label{th-DN}
A Fr\'echet space is isomorphic to:
\begin{enumerate}[\upshape(i)]
 \item a closed subspace of $s$ if and only if it is nuclear and has the property \emph{(DN)};
 \item a quotient of $s$ if and only if it is nuclear and has the property \emph{$(\Omega)$};
 \item a complemented subspace of $s$ if and only if it is nuclear and has the properties \emph{(DN)} and \emph{$(\Omega)$}.
\end{enumerate} 
\end{theorem}

We also cite another result of Vogt which will be crucial for our futher considerations.  
\begin{theorem}\label{th-vogt-unitary-iso}\cite[Cor. 7.7]{V3}
Let $E$ be a Fr\'echet space isomorphic to a power series space $\Lambda_\infty(\alpha)$ of infinite type. Then for every dominating Hilbert norm $||\cdot||$ on $E$ there is an isomorphism 
$u\colon E\to\Lambda_\infty(\alpha)$ such that $||u\xi||_{\ell_2}=||\xi||$ for all $\xi\in E$. 
\end{theorem}

Let $E$ be a Fr\'echet space with a continuous Hilbert norm $||\cdot||$. Let $H$ be the completion of $E$ in the norm $||\cdot||$ and let $(\cdot,\cdot)$ be the corresponding scalar product.
Then we define
\[\cL^*(E,||\cdot||):=\{x\colon E\to E: x\textrm{ is linear},E\subset\mathscr{D}(x^*)\textrm{ and }x^*(E)\subset E\},\]
where
\[\cD(x^*):=\{\eta\in H:\;\exists\zeta\in H\;\forall\xi\in E\quad(x\xi,\eta)=(\xi,\zeta)\}\]
and $x^*\eta:=\zeta$ for $\eta\in\cD(x^*)$.
In the case when $E$ is a closed subspace of $s$ or $E=\Lambda_\infty(\alpha)$ we write $\cL^*(E)$ instead of $\cL^*(E,||\cdot||_{\ell_2})$.
Since $E$ is a dense linear subspace of $H$, each $x\in\cL^*(E,||\cdot||)$ can be considered as a dense unbounded operator in $H$ with domain $\cD(x)=E$, and thus it has the adjoint $x^*\colon \cD(x^*)\to H$.
By definition, the operator ${x^*}_{\mid{E}}$, for simplicity denoted again by $x^*$, is in $\cL^*(E,||\cdot||)$, as well. Moreover, by definition, 
\[\cD(xy):=\{\xi\in\cD(y):\;y\xi\in\cD(x)\}=E\]
for all $x,y\in\cL^*(E,||\cdot||)$. This shows that $\cL^*(E,||\cdot||)$ is a ${}^*$-algebra. 
In fact, the class $\cL^*(E,||\cdot||)$ can be seen as the largest ${}^*$-algebra of unbounded operators on $H$ with domain $E$ and it is known as the \emph{maximal $\mathcal{O}^*$-algebra} with domain $E$
(see \cite[2.1]{Sch} for details). 

In the theory of maximal $\mathcal{O}^*$-algebras -- and, more generally, of algebras of unbounded operators in Hilbert spaces -- one consider the so-called \emph{graph topology} (\cite[Def. 2.1.1]{Sch}). 
With $E$ and $||\cdot||$ as above, the graph topology of $\cL^*(E,||\cdot||)$ on $E$ is, by definition, given by the system of seminorms 
$(||\cdot||_a)_{a\in\cL^*(E,||\cdot||)}$, $||\xi||_a:=||a\xi||$ for $\xi\in E$. 

The following easy observation is kind of folklor -- for completness we present here the proof. 
\begin{proposition}\label{prop-graph-topology}
Let $E$ be a Fr\'echet space with a continuous Hilbert norm $||\cdot||$. 
Then the graph topology of $\cL^*(E,||\cdot||)$ on $E$ is weaker than the Fr\'echet space toplogy.
\end{proposition}
\begin{proof}
Let $(\cdot,\cdot)$ denote the scalar product corresponding to the Hilbert norm $||\cdot||$ and let $H$ be the completion of $E$ in the norm $||\cdot||$.
We shall show that each $a\in\cL^*(E,||\cdot||)$ is a continuous map from the Fr\'echet space $E$ to the Hilbert space $H$.
Let $(\xi_j)_{j\in\N}\subset E$ be a sequence converging in the Fr\'echet space topology to 0 and assume that $a\xi_j$ converges in the norm 
$||\cdot||$ to some $\eta\in H$.
We have, for all $\zeta\in E$, 
\[\lim_{j\to\infty}(a\xi_j,\zeta)=(\eta,\zeta)\]
and, on the other hand,
\[\lim_{j\to\infty}(a\xi_j,\zeta)=\lim_{j\to\infty}(\xi_j,a^*\zeta)=0.\]
Hence, $(\eta,\zeta)=0$ for all $\zeta\in E$, and thus $\eta=0$. Consequently, by the closed graph theorem for Fr\'echet spaces 
(cf. \cite[Th. 24.31]{MeV}), the map $a\colon E\to H$ is continuous, which is the desired conclusion.   
\end{proof}

Sometimes the initial Fr\'echet space topology and the graph topology coincide.
\begin{proposition}\label{prop-top-power-series}
Let $E$ be a Fr\'echet space isomorphic to a power series space $\Lambda_\infty(\alpha)$ of infinite type and let $||\cdot||$ be a dominating Hilbert norm on $E$. 
Then the graph topology of $\cL^*(E,||\cdot||)$ on $E$ coincides with the Fr\'echet space topology.
\end{proposition}
\begin{proof}
Let $(\cdot,\cdot)$ denote the scalar product corresponding to the Hilbert norm $||\cdot||$.
By \cite[Cor. 7.7]{V3}, there is an isomorphism $u\colon E\to\Lambda_\infty(\alpha)$ such that $||u \xi||_{\ell_2}=||\xi||$ for all $\xi\in E$.
Let $||\xi||_n:=|u \xi|_{\alpha,n}$ for $\xi\in E$ and $n\in\N$. Then $(||\cdot||_n)_{n\in\N}$ is a fundamental sequence of dominating Hilbert norms on $E$.
For $n\in\N$, we define the diagonal map $d_n\colon\Lambda_\infty(\alpha)\to\Lambda_\infty(\alpha)$, $d_n\xi:=(e^{n\alpha_j}\xi_j)_{j\in\N}$.   
Clearly, each $d_n$ is an automorphism of the Fr\'echet space $\Lambda_\infty(\alpha)$ and $||d_n\xi||_{\ell_2}=|\xi|_{\alpha,n}$ for all
$\xi\in\Lambda_\infty(\alpha)$. Now, for $n\in\N$, let $a_n\colon E\to E$, $a_n:=u^{-1}d_nu$. We have
\[(a_n\xi,\zeta)=(u^{-1}d_nu\xi,\zeta)=\langle d_nu\xi,u\zeta\rangle=\langle u \xi,d_nu \zeta\rangle=(\xi,u ^{-1}d_nu\zeta)=(\xi,a_n\zeta)\]
for all $\xi,\zeta\in E$, whence $a_n\in \cL^*(E,||\cdot||)$. Consequently, since $||\xi||_n=||a_n\xi||$ for all $\xi\in E$, i.e. 
$||\cdot||_n=||\cdot||_{a_n}$, the graph topology of $\cL^*(E,||\cdot||)$ on $E$ is finer than the Fr\'echet space toppology
and thus, in view of Proposition \ref{prop-graph-topology}, these topologies are equal. 
\end{proof}

There are plenty natural topologies on the space $\cL^*(E,||\cdot||)$ (see \cite[Sect. 3.3, 3.5]{Sch}).  
Here we are interested in the locally convex topology $\tau^*$ on $\cL^*(E,||\cdot||)$ given by the seminorms  
\[p^{a,B}(x):=\max\Big\{\sup_{\xi\in B}||ax\xi||,\,\sup_{\xi\in B}||ax^*\xi||\Big\},\]
where $a$ and $B$ run over $\cL^*(E,||\cdot||)$ and the class of all bounded subsets of $E$ equipped with the graph topology of $\cL^*(E,||\cdot||)$, 
respectively (see \cite[pp. 81--82]{Sch}). 
It is well-known that $\cL^*(E,||\cdot||)$ endowed with the topology $\tau^*$ is a topological ${}^*$-algebra 
(cf. \cite[Prop. 3.3.15 (i)]{Sch}). If we, moreover, assume that $E$ is isomorphic to a power series space of infinite type and $||\cdot||$ is a dominating Hilbert norm on $E$,
then, by Proposition \ref{prop-top-power-series} and \cite[Prop. 3.3.15 (iv)]{Sch}, $(\cL^*(E,||\cdot||),\tau^*)$ is complete. 
The following characterization of the topology $\tau^*$ is a direct consequence of Proposition \ref{prop-top-power-series}.

\begin{proposition}\label{prop-top-Lstar-power-series}
Let $E$ be a Fr\'echet space isomorphic to a power series space of infinite type and let $||\cdot||$ be a dominating Hilbert norm on $E$.
Let $(||\cdot||_n)_{n\in\N}$ be a fundamental sequence of norms on $E$.
Then the topology $\tau^*$ on $\cL^*(E,||\cdot||)$ is given by the seminorms $(p_{n,B})_{n\in\N,B\in\cB_E}$,
\begin{equation}\label{pnB}
p_{n,B}(x):=\max\{\sup_{\xi\in B}||x\xi||_n,\sup_{\xi\in B}||x^*\xi||_n\},
\end{equation}
where $\cB_E$ denote the class of all bounded subsets of $E$.
\end{proposition}

\section{Fr\'echet subalgebras of $\lstars$}

In this section we give abstract descriptions of two large classes of complemented commutative Fr\'echet ${}^*$-subalgebras of $\lstars$ 
(Theorems \ref{th_complemented_subalg_lstars} and \ref{th-complemented-bounded}).  
Moreover, we provide a criterion for the existence of a ``nice'' embedding in $\lstars$ of not necessarily commutative Fr\'echet ${}^*$-algebras (see Remark \ref{rem-noncommutative}).

Let us first recall the notion of Hilbert algebras.
\begin{definition}\label{def-Hilbert-algebra}(cf. \cite[A.54]{Dix77})
A \emph{Hilbert algebra} is a ${}^*$-algebra $E$ endowed with a Hilbert norm $||\cdot||:=\sqrt{(\cdot,\cdot)}$ such that:
\begin{enumerate}
 \item[\upshape($\alpha$)] $(xy,z)=(y,x^*z)$ for all $x,y,z\in E$;
 \item[\upshape($\beta$)] for all $x\in E$ there is $C>0$ such that $||xy||\leq C||y||$ for all $y\in E$, i.e. the left multiplication maps 
 $m_x\colon (E,||\cdot||)\to (E,||\cdot||)$, $m_x(y):=xy$, are bounded;
 \item[\upshape($\gamma$)] $(y^*,x^*)=(x,y)$ for all $x,y\in E$;
 \item[\upshape($\delta$)] the linear span of the set $E^2:=\{ab:\; a,b\in E\}$ is dense in $E$. 
\end{enumerate}
Each norm $||\cdot||$ satisfying conditions ($\alpha$)--($\delta$) is called a \emph{Hilbert algebra norm}.
\end{definition}

\begin{remark}\label{rem-def-Hilbert-alg}
If $E$ is unital, then condition ($\delta$) in the above definition is trivially satisfied. If, moreover, 
$E$ is commutative, then ($\alpha$) implies ($\gamma$). Hence, every Hilbert norm on a unital commutative ${}^*$-algebra satisfying condition ($\alpha$) and ($\beta$) is already a Hilbert algebra norm.  
\end{remark}

\begin{definition}
A Fr\'echet ${}^*$-algebra is called a DN-\emph{algebra} if it admits a Hilbert dominating norm satysfying condition ($\alpha$) in Definition \ref{def-Hilbert-algebra}.
A DN-algebra is called a $\beta\mathrm{DN}$-\emph{algebra} if the corresponding Hilbert dominating norm satisfies conditions ($\alpha$) and ($\beta$) simultaneously.
\end{definition}

\begin{remark}
In \cite[Def. 1.5]{MP14} M. M\u antoiu and R. Purice defined a Fr\'echet-Hilbert algebra as a Fr\'echet ${}^*$-algebra 
admitting a continous Hilbert algebra norm (more precisely, in their defintion the corresponding Hilbert algebra scalar product is predetermined).
Hence, in view of Remark \ref{rem-def-Hilbert-alg}, every unital commutative $\beta$DN-algebra is a Fr\'echet-Hilbert algebra.
\end{remark}

Our main results read as follows.

\begin{theorem}\label{th_complemented_subalg_lstars}
Let $E$ be a unital commutative Fr\'echet ${}^*$-algebra isomorphic as a Fr\'echet space to a nuclear power series space of infinite type. Then the following statements are equivalent.
\begin{enumerate}[\upshape(i)]
 \item $E$ is isomorphic as a Fr\'echet ${}^*$-algebra to a complemented ${}^*$-subalgebra of $\lstars$.
 \item $E$ is isomorphic as a Fr\'echet ${}^*$-algebra to a closed ${}^*$-subalgebra of $\lstars$.
 \item $E$ is a $\mathrm{DN}$-algebra. 
\end{enumerate}
\end{theorem}

\begin{theorem}\label{th-complemented-bounded}
Let $E$ be a unital commutative Fr\'echet ${}^*$-algebra isomorphic as a Fr\'echet space to a nuclear power series space of infinite type. Then the following statements are equivalent.
\begin{enumerate}[\upshape(i)]
 \item $E$ is isomorphic as a Fr\'echet ${}^*$-algebra to a complemented ${}^*$-subalgebra $F$ of $\lstars$ such that $F\subset \cL(\ell_2)$.
 \item $E$ is isomorphic as a Fr\'echet ${}^*$-algebra to a closed ${}^*$-subalgebra $F$ of $\lstars$ such that $F\subset \cL(\ell_2)$.
 \item $E$ is a $\beta\mathrm{DN}$-algebra.
\end{enumerate}
\end{theorem}

We divide the proof into a sequence of lemmas. As a by-product, we obtain also three results which are interesting enough to be stated as
``corollaries".

For every $N,n\in\N_0$ we define the space
\[\cL(s_n,s_N)\cap\cL(s_{-N},s_{-n}):=\{x\in\cL(s_n,s_N):\;\exists\widetilde{x}\in\cL(s_{-N},s_{-n})\quad \widetilde{x}_{\mid s_n}=x\}\]
with the norm
\[r_{N,n}(x):=\max\bigg\{\sup_{|\xi|_{n}\leq1}|x\xi|_N,\sup_{|\xi|_{-N}\leq1}|\widetilde{x}\xi|_{-n}\bigg\}.\]
Formally, the space $\cL(s_n,s_N)\cap\cL(s_{-N},s_{-n})$ is the projective limit of the Banach spaces $\cL(s_n,s_N)$ and $\cL(s_{-N},s_{-n})$ 
with their standard norms, and thus it is a Banach space itself.
Since 
\[\sup_{|\xi|_{n}\leq1}|x^*\xi|_N=\sup_{|\xi|_{n}\leq1}\sup_{|\eta|_{-N}\leq1}|\langle x^*\xi,\eta\rangle|=\sup_{|\xi|_{n}\leq1}\sup_{|\eta|_{-N}\leq1}|\langle \xi,\widetilde{x}\eta\rangle|
=\sup_{|\eta|_{-N}\leq1}|\widetilde{x}\eta|_{-n},\]
we have
\[r_{N,n}(x)=\max\bigg\{\sup_{|\xi|_{n}\leq1}|x\xi|_N,\sup_{|\xi|_{n}\leq1}|x^*\xi|_N\bigg\},\]
where $x^*\in\cL(s_n,s_N)$ is the hilbertian adjoint of the operator $\widetilde{x}$. 
Moreover, $\lstars=\lss\cap\ldsds$ (see, e.g., \cite[Prop. 3.7]{CiasPisz}), hence  
\[\lstars=\{x\colon s\to s:\;x\text{ linear and }\forall N\in\N_0\,\exists n\in\N_0\quad r_{N,n}(x)<\infty\}\]
as sets. Therefore, we can endow $\lstars$ with the topology of the PLB-space (a countable projective limit of a countable inductive limit of Banach spaces) 
\[\operatorname{proj}_{N\in\N_0}\operatorname{ind}_{n\in\N_0}\cL(s_n,s_N)\cap\cL(s_{-N},s_{-n}).\]
It appears that the topology $\tau^*$ and the PLB-topology on $\lstars$ coincide.  

\begin{lemma}
We have
\[\lstars=\operatorname{proj}_{N\in\N_0}\operatorname{ind}_{n\in\N_0}\cL(s_n,s_N)\cap\cL(s_{-N},s_{-n})\]
as topological vector spaces.
\end{lemma}
\begin{proof}
By \cite[Cor. 4.2]{CiasPisz}, $\lstars$ is ultrabornological and $\operatorname{proj}_{N\in\N_0}\operatorname{ind}_{n\in\N_0}\cL(s_n,s_N)\cap\cL(s_{-N},s_{-n})$ is webbed as a PLB-space. 
Hence, by the open mapping theorem (see e.g. \cite[Th. 24.30]{MeV}), it is enough to show that the identity map 
\[\iota\colon\operatorname{proj}_{N\in\N_0}\operatorname{ind}_{n\in\N_0}\cL(s_n,s_N)\cap\cL(s_{-N},s_{-n})\to\lstars\]
is continuous. Let $N\in\N_0$ and let $B$ be a bounded subset of $s$. For every $m\in\N_0$ choose a constant $\lambda_m>0$ such that 
$B\subset \{\xi\in s\colon |\xi|_m\leq \lambda_m\}$. Then 
\[p_{N,B}(x)=\max\bigg\{\sup_{\xi\in B}|x\xi|_N,\sup_{\xi\in B}|x^*\xi|_N\bigg\}\leq\lambda_m\max\bigg\{\sup_{|\xi|_m\leq1}|x\xi|_N,\sup_{|\xi|_m\leq1}|x^*\xi|_N\bigg\}=\lambda_m r_{N,m}(x)\]
for every $m\in\N_0$ and $x\in\cL(s_m,s_N)\cap\cL(s_{-N},s_{-m})$, and thus $\iota$ is continuous.
\end{proof}

\begin{lemma}\label{lem_frechet_subspaces_lstars}
For every Fr\'echet subspace $F$ of $\lstars$ there is $m\in\N_0$ such that $F\subset\cL(s_m,\ell_2)\cap\cL(\ell_2,s_{-m})$ and, moreover, for each such $m$,
\[r_m\colon F\to [0,\infty),\quad r_m(x):=\max\{\sup_{|\xi|_{m}\leq1}||x\xi||_{\ell_2},\sup_{|\xi|_{m}\leq1}||x^*\xi||_{\ell_2}\},\]
is a dominating norm on $F$.
\end{lemma}
\begin{proof}
By the very definition of projective topology, the canonical embedding 
\[\operatorname{proj}_{N\in\N_0}\operatorname{ind}_{n\in\N_0}\cL(s_n,s_N)\cap\cL(s_{-N},s_{-n})\hookrightarrow\operatorname{ind}_{n\in\N_0}\cL(s_n,\ell_2)\cap\cL(\ell_2,s_{-n})\]
is continuous and thus the identity map
$\kappa\colon F\hookrightarrow\operatorname{ind}_{n\in\N_0}\cL(s_n,\ell_2)\cap\cL(\ell_2,s_{-n})$ is continuous, as well. 
Hence, by Grothendieck's factorization theorem \cite[Th. 24.33]{MeV}, there is $m\in\N$ such that 
$\kappa(F)\subset \cL(s_m,\ell_2)\cap\cL(\ell_2,s_{-m})$. Since we can identify in a obvious way $F$ with $\kappa(F)$, we get the first
part of the thesis.

Now, fix an arbitrary $m\in\N_0$ such that $F\subset\cL(s_m,\ell_2)\cap\cL(\ell_2,s_{-m})$. Then $r_m$ is a continuous seminorm on $F$. 
Since $F$ is a Fr\'echet space, there is a sequence $(B_N)_{N\in\N}$, $B_N\subset B_{N+1}$, of bounded subsets of $s$ such that $({p_{N}})_{N\in\N}$,
\[p_{N}(x):=\max\bigg\{\sup_{\xi\in B_N}|x\xi|_N,\sup_{\xi\in B_N}|x^*\xi|_N\bigg\}\]
for $x\in F$, is a fundamental sequence of seminorms on $F$.
Moreover, for every $N\in\N$ there is $\lambda_N>0$ such that 
$B_N\subset\{\xi\in s:\,\,|\xi|_{m}\leq \lambda_N\}$. 
Hence, for $x\in F$ and $N\in\N$, we obtain
\begin{align*}
p_{N}^2(x) &=\max\bigg\{\sup_{\xi\in B_N}|x\xi|_N^2,\sup_{\xi\in B_N}|x^*\xi|_N^2\bigg\}
\leq\max\bigg\{\sup_{\xi\in B_N}(||x\xi||_{\ell_2}|x\xi|_{2N}),\sup_{\xi\in B_N}(||x^*\xi||_{\ell_2}|x^*\xi|_{2N})\bigg\}\\
&\leq\max\bigg\{\sup_{\xi\in B_N}||x\xi||_{\ell_2}\cdot\sup_{\xi\in B_N}|x\xi|_{2N},\sup_{\xi\in B_N}||x^*\xi||_{\ell_2}\cdot\sup_{\xi\in B_N}|x^*\xi|_{2N}\bigg\}\\
&\leq \lambda_N\max\bigg\{\sup_{|\xi|_{m}\leq1}||x\xi||_{\ell_2}\cdot\sup_{\xi\in B_{2N}}|x\xi|_{2N},\sup_{|\xi|_{m}\leq1}||x^*\xi||_{\ell_2}\cdot\sup_{\xi\in B_{2N}}|x^*\xi|_{2N}\bigg\},
\end{align*}
where the first inequality follows from the Cauchy-Schwartz inequality.
Finally, since
\[\max\{ab,cd\}\leq\max\{a,c\}\cdot\max\{b,d\}\]
for all $a,b,c,d\geq0$, we obtain
\begin{align*}
p_{N}^2(x)&\leq \lambda_N\max\bigg\{\sup_{|\xi|_{m}\leq1}||x\xi||_{\ell_2},\sup_{|\xi|_{m}\leq1}||x^*\xi||_{\ell_2}\bigg\}\cdot\max\bigg\{\sup_{\xi\in B_{2N}}|x\xi|_{2N},\sup_{\xi\in B_{2N}}|x^*\xi|_{2N}\bigg\}\\
&=\lambda_Nr_m(x)p_{2N}(x)
\end{align*}
for all $x\in F$, and thus $r_m$ is a dominating norm on $F$.
\end{proof}

\begin{corollary}\label{cor-frechet-subspace-quotient}
\begin{enumerate}[\upshape(i)]
 \item Every Fr\'echet subspace of $\lstars$ is isomorphic to a closed subspace of $s$. 
 \item Every Fr\'echet quotient of $\lstars$ is isomorphic to a quotient of $s$.
 \item Every complemented Fr\'echet subspace of $\lstars$ is isomorphic to a complemented subspace of $s$. 
\end{enumerate}
\end{corollary}
\begin{proof}
First note that every closed subspace and quotient of $\lstars$ is nuclear because $\lstars$ is nuclear itself (see \cite[Prop. 3.8 \& Cor. 4.2]{CiasPisz}). 

(i) This follows immediately from Lemma \ref{lem_frechet_subspaces_lstars} and \cite[Prop. 31.5]{MeV}.

(ii) Let $E$ be a Fr\'echet quotient of $\lstars$. It follows from \cite[Prop. 4.7]{CiasPisz} and \cite[Cor. 1.2(a) and (c)]{BonDom} that $E$, being isomorphic to a quotient of $\lstars$, has the property (\textrm{$\Omega$}). 
Therefore, by \cite[Prop. 31.6]{MeV}, $E$ is isomorphic to a quotient of $s$.

(iii) This is a direct consequence of the previous items and \cite[Prop. 31.7]{MeV}.  
\end{proof}

Let $e_j$ denote the $j$-th unit vector in $\C^\N$.
If $F$ is a Fr\'echet subspace of $\lstars$ then, by Lemma \ref{lem_frechet_subspaces_lstars}, there is $m\in\N_0$ such that $F\subset\cL(s_m,\ell_2)\cap\cL(\ell_2,s_{-m})$ and $r_m$ 
is a continuous (dominating) norm on $F$. Since, for all $x\in F$, we have 
\begin{align*}
[x]_m&:=\bigg(\sum_{j=1}^\infty||xe_j||_{\ell_2}^2j^{-2m-2}\bigg)^{1/2}
\leq\bigg(\sum_{j=1}^\infty j^{-2}\bigg)^{1/2}\cdot\sup_{|\xi|_m\leq1}||x\xi||_{\ell_2}\\
&\leq \frac{\pi}{\sqrt{6}}\max\bigg\{\sup_{|\xi|_m\leq1}||x\xi||_{\ell_2},\sup_{|\xi|_m\leq1}||x^*\xi||_{\ell_2}\bigg\}=\frac{\pi}{\sqrt{6}}r_m(x),
\end{align*}
the scalar product  
\begin{equation}\label{eq_scalar_product_Fm}
[\cdot,\cdot]_m\colon F\times F\to\C,\quad[x,y]_m:=\sum_{j=1}^\infty \langle xe_j,ye_j\rangle j^{-2m-2}, 
\end{equation}
is well-defined and $[\;\cdot\;]_m=\sqrt{[\cdot,\cdot]_m}$ is a continuous Hilbert norm on $F$.

\begin{lemma}\label{lem_commutative_subalgebras_lstars}
Let $F$ be a commutative Fr\'echet ${}^*$-subalgebra of $\lstars$ and let $m\in\N_0$ be such that $F\subset\cL(s_m,\ell_2)\cap\cL(\ell_2,s_{-m})$. 
Then the norm $[\;\cdot\;]_m$ defined by \textnormal{(\ref{eq_scalar_product_Fm})} is a Hilbert dominating norm on $F$ satisfying condition \textnormal{($\alpha$)}.
\end{lemma}
\begin{proof}
Since $F$ is commutative, we have
\[||x\xi||_{\ell_2}=||x^*\xi||_{\ell_2}\]
for all $x\in E$ and all $\xi\in s$. Hence,
\begin{align*}
r_{m+2}(x)&=\max\bigg\{\sup_{|\xi|_{m+2}\leq1}||x\xi||_{\ell_2},\sup_{|\xi|_{m+2}\leq1}||x^*\xi||_{\ell_2}\bigg\}=\sup_{|\xi|_{m+2}\leq1}||x\xi||_{\ell_2}\\
&=\sup_{|\xi|_{m+2}\leq1}\big|\big|\sum_{j=1}^\infty\xi_jj^{m+2}\cdot x(e_jj^{-m-2})\big|\big|_{\ell_2}\leq\sup_{|\xi|_{m+2}\leq1}\sum_{j=1}^\infty|\xi_j|j^{m+2}\cdot||xe_j||_{\ell_2}\cdot j^{-m-2}\\
&\leq\sum_{j=1}^\infty||xe_j||_{\ell_2}\cdot j^{-m-2}
\leq\bigg(\sum_{j=1}^\infty j^{-2}\bigg)^{1/2}\cdot\bigg(\sum_{j=1}^\infty||xe_j||_{\ell_2}^2j^{-2m-2}\bigg)^{1/2}=\frac{\pi}{\sqrt{6}}[x]_m.
\end{align*}
Therefore, 
\[[\;\cdot\;]_m\geq \frac{\sqrt{6}}{\pi}r_{m+2},\]
and, by Lemma \ref{lem_frechet_subspaces_lstars}, $[\cdot]_m$ is a dominating norm on $F$. 
Moreover, we have
\[[xy,z]_m=\sum_{j=1}^\infty\langle xye_j,ze_j\rangle j^{-2m-2}=\sum_{j=1}^\infty\langle ye_j,x^*ze_j\rangle j^{-2m-2}=[y,x^*z]_m,\]
which completes the proof.
\end{proof}

\begin{definition}
A closed subspace $E$ of the space $s$ is called \emph{orthogonally complemented} in $s$ if there is a continuous projection $\pi$ in $s$ onto $E$ 
admitting the extension to the orthogonal projection in $\ell_2$.
Then we call $\pi$ an \emph{orthogonal projection} in $s$ onto $E$.   
\end{definition}

\begin{lemma}\label{lemma_iso_onto_orth_complemented}
Let $E$ be a Fr\'echet space isomorphic to a nuclear power series space of infinite type and let $||\cdot||$ be a dominating Hilbert norm on $E$. Then there is an orthogonally complemented subspace $G$
of $s$ and an isomorphism $w\colon E\to G$ of Fr\'echet spaces such that $||w\xi||_{\ell_2}=||\xi||$ for all $\xi\in E$.
\end{lemma}
\begin{proof}
Since $E$ is isomorphic to a nuclear power series space of infinite type, by \cite[Lemma 29.2(3) $\&$ Lemma 29.11(3)]{MeV}, $E$ has 
the properties (DN) and ($\Omega$). Hence, by \cite[Prop. 31.7]{MeV}, $E$ is isomorphic to a complemented subspace of $s$.
This means that there is a complemented subspace $F$ of $s$ with a continuous projection $\pi\colon s\to F$ and a Fr\'echet space isomorphism 
$\psi\colon E\to F$. 
Hence, $||\cdot||_\psi\colon F\to[0,\infty)$ defined by $||\xi||_\psi:=||\psi^{-1}\xi||$ is a dominating Hilbert norm on $F$. 
Since, $||\cdot||_{\ell_2}$ is also a dominating Hilbert norm on $F$, by \cite[Cor. 7.7]{V3}, there is an automorphism $u$ of $F$ such that 
$||u\xi||_{\ell_2}=||\xi||_\psi$ for all $\xi\in F$.
Moreover, by \cite[Th. 7.2]{V3}, there is an automorphism $v$ of $s$ such that $\rho:=v\pi v^{-1}$ is the orthogonal projection in $s$ 
onto $G:=v(F)$ and a simple analysis of the proof of \cite[Th. 7.2]{V3} shows that $||v\xi||_{\ell_2}=||\xi||_{\ell_2}$ for all $\xi\in F$. 
Therefore, the operator $w:=vu\psi$ has the desired properties. 
\end{proof}

\begin{lemma}\label{lemma_lstarE}
Let $E$ be a Fr\'echet space isomorphic to a nuclear power series space of infinite type and let $||\cdot||$ be a dominating Hilbert norm on $E$.  
Let $H$ denote the completion of $E$ in the norm $||\cdot||$.  
Then there is a map $\phi\in\cL(H,\ell_2)$ and an orthogonally complemented subspace $G$ of $s$ such that 
\begin{enumerate}[\upshape(i)]
 \item $\phi(E)=G$;
 \item $\phi^*(s)=E$;
 \item $||\phi\xi||_{\ell_2}=||\xi||$ for all $\xi\in E$;
 \item ${\phi\phi^*}$ is the orthogonal projection in $\ell_2$ with $\phi\phi^*(s)=G$. 
\end{enumerate}
Moreover, the map 
\[\Phi\colon\cL^*(E,||\cdot||)\to\lstars,\quad x\mapsto \phi x\phi^*,\] 
is a continuous injective ${}^*$-algebra homomorphism with $\im\Phi=\cL^*(G)$ and the map 
\[P\colon\lstars\to\cL^*(G),\quad x\mapsto \phi\phi^*x\phi\phi^*,\] 
is a continuous projection onto $\cL^*(G)$.
\end{lemma}
\begin{proof}
By Lemma \ref{lemma_iso_onto_orth_complemented}, there is an orthogonally complemented subspace $G$ of $s$ and an isomorphism $w\colon E\to G$ of Fr\'echet spaces such that $||w\xi||_{\ell_2}=||\xi||$ 
for all $\xi\in E$. Let $\rho\colon s\to s$ be the orthogonal projection onto $G$. The operators $w$, $\rho$ and the identity map $\iota\colon G\hookrightarrow s$ can be extended to the continuous linear operators 
between Hilbert spaces (for simplicity denoted by the same symbols): $w\colon H\to\overline{G}$, $\rho\colon\ell_2\to\ell_2$ and $\iota\colon\overline{G}\hookrightarrow\ell_2$, 
where $\overline{G}$ is the closure of $G$ in $\ell_2$.
Therefore, the Hermitian adjoints $w^*$ and $\iota^*$ of the operators $w$ and $\iota$ are well-defined. 
We have thus the following commutative diagram of continuous linear maps between Fr\'echet and Hilbert spaces

\begin{displaymath}
\xymatrix{
E             \ar @{->}[r]^w \ar@{^{(}->}[d] & G             \ar@{^{(}->}[r]^\iota\ar@{^{(}->}[d] & s        \ar@{^{(}->}[d]\\
H \ar @{->}[r]^w                 & \overline{G} \ar@{^{(}->}[r]^\iota                & \ell_2 &
}  
\end{displaymath}
and the diagram with the corresponding adjoint operators
\begin{displaymath}
\xymatrix{
\ell_2 \ar @{->}[r]^{\iota^*} \ar@{->}[rd]^{\rho} & \overline{G} \ar@{^{(}->}[d]^\iota \ar@{->}[r]^{w^*} & H\\
{}                                            & \ell_2.
}  
\end{displaymath}
It follows easily that $\iota^*\colon\ell_2\to \overline{G}$ is the orthogonal projection onto $\overline{G}$, whence
$\iota^*(s)=\rho(s)=G$. Moreover, $w^*(G)=E$. Indeed, if $(\cdot,\cdot)$ denotes the scalar product on $E$ corresponding to the Hilbert norm $||\cdot||$,
then
\[(w^*w\xi,\eta)=\langle  w\xi,w\eta\rangle=(\xi,\eta)\]
for all $\xi,\eta\in E$. 
Hence, $E$ being dense in $H$, $w^*w=\id_{H}$, and so $w^*(G)=E$. 
Consequently, we have the following commutative diagram
\begin{displaymath}
\xymatrix{
s \ar @{->}[r]^{\iota^*} \ar@{->}[rd]^{\rho} & G \ar@{^{(}->}[d]^\iota \ar@{->}[r]^{w^*} & E\\
{}                                            & s.
}  
\end{displaymath}

It is easy to check that $\phi:=\iota w$ satisfies conditions (i)--(iii) and a simple computation shows that $\phi\phi^*$ is a self-adjoint projection 
(and thus orthogonal) on $\ell_2$ with $\phi\phi^*(s)=G$. 
In consequence, $\Phi\colon\cL^*(E,||\cdot||)\to\lstars, x\mapsto \phi x\phi^*$, is an injective ${}^*$-homomorphism with $\im\Phi=\cL^*(G)$ and, moreover, 
$P\colon\lstars\to\cL^*(G), x\mapsto \phi\phi^*x\phi\phi^*$, is a projection. 

Now, we shall prove the continuity of $\Phi$.
Let $B$ be a bounded subset of $s$ and let $n\in\N_0$. By the closed graph theorem, $\phi\colon E\to s$ and $\phi^*\colon s\to E$ are continuous maps between Fr\'echet spaces.
Hence, there is a constant $C>0$ and a continuous norm 
$||\cdot||_E$ on $E$ such that $|\phi\xi|_n\leq C||\xi||_E$ for $\xi\in E$. Note also that the set $\phi^*(B)$ is bounded in the Fr\'echet space $E$. 
Therefore,
\[\max_{\xi\in B}\{|(\Phi x)\xi|_n,|(\Phi x)^*\xi|_n\}=\max_{\xi\in B}\{|\phi x\phi^*\xi|_n,|\phi x^*\phi^*\xi|_n\}\leq C\max_{\eta\in \phi^*(B)}\{||x\eta||_E,||x^*\eta||_E\},\]
which, by Proposition \ref{prop-top-Lstar-power-series}, gives the continuity of $\Phi$. The continuity of $P$ can be proved in a simillar way.
\end{proof}

\begin{corollary}\label{cor_lstarF_subset_lstars}
For every Fr\'echet space $E$ isomorphic to a nuclear power series space of infinite type and a dominating Hilbert norm $||\cdot||$ on $E$ there is an orthogonally complemented subspace $G$ of $s$ such that
$\cL^*(E,||\cdot||)\cong\cL^*(G)$ as topological ${}^*$-algebras. Moreover, the algebra $\cL^*(G)$ is a complemented ${}^*$-subalgebra of $\lstars$. 
\end{corollary}
\begin{proof}
This follows directly from Lemma \ref{lemma_lstarE}. 
\end{proof}

In our next corollary we deal with PLS-spaces, i.e. countable projective limits of strong duals of Fr\'echet-Schwartz spaces (see \cite{Dom04} for basic properties and examples).
\begin{corollary}\label{cor_lstarE_nuclear_ultra_PLS}
Let $E$ be a Fr\'echet space isomorphic to a nuclear power series space of infinite type and let $||\cdot||$ be a dominating Hilbert norm on $E$.
Then the space $\cL^*(E,||\cdot||)$ is a nuclear, ultrabornological $\mathrm{PLS}$-space. 
\end{corollary}
\begin{proof}
By \cite[Prop. 3.8 \& Cor. 4.2]{CiasPisz}, $\lstars $ is a nuclear, ultrabornological $\mathrm{PLS}$-space. These properties are inherited by complemented subspaces 
(see \cite[Prop. 28.6]{MeV}, \cite[Ch. II, 8.2, Cor. 1]{HSch}, \cite[Prop. 1.2]{DomV}). 
Hence, the desired conclusion follows from Corollary \ref{cor_lstarF_subset_lstars}. 
\end{proof}

\begin{lemma}\label{lem-M_E}
Let $E$ be a unital \textnormal{DN}-algebra isomorphic as a Fr\'echet space to a nuclear power series space of infinite type
and let $||\cdot||:=\sqrt{(\cdot,\cdot)}$ be the corresponding Hilbert norm. 
Let  
\[\cM_E:=\{m_x\colon x\in E \},\]
where $m_x\colon E\to E$, $m_xy:=xy$, denotes the left multiplication map for the element $x$.
Then $\cM_E$ is a complemented ${}^*$-subalgebra of $\cL^*(E,||\cdot||)$ and $E$ is isomorphic as a Fr\'echet ${}^*$-algebra to $\cM_E$. 
\end{lemma}
\begin{proof}
By assumption, 
\[(m_xy,z)=(xy,z)=(y,x^*z)=(y,m_{x^*}z)\]
for all $x,y,z\in E$, hence $E\subset\cD((m_x)^*)$ and ${(m_x)^*}_{\mid{E}}=m_{x^*}$. Consequently, $m_x\in\cL^*(E,||\cdot||)$ for all $x\in E$. 

Define $Q\colon\cL^*(E,||\cdot||)\to\cL^*(E,||\cdot||)$ by $Q\phi:=m_{\phi(\mathbf{1})}$, where $\mathbf{1}$ is the unit in $E$. Clearly, $Q$ is a projection onto $\cM_E$; we will show that $Q$ is continuous.
By Corollary \ref{cor_lstarE_nuclear_ultra_PLS} and the closed graph theorem (see e.g. \cite[Th. 24.31]{MeV}), every linear map on $\cL^*(E,||\cdot||)$ with closed graph is continuous. 
Assume that a net $(\phi_\lambda)_\lambda\subset\cL^*(E,||\cdot||)$ converges to 0,
$(Q\phi_\lambda)_{\lambda}$ converges to $\psi$ and both limits are taken in $\cL^*(E,||\cdot||)$. Let us fix $x\in E$.
By the continuity of the multiplication in $E$, there is $C>0$ and a continuous norm $||\cdot||_1$  on $E$ with $||yx||\leq C||y||_1$ for all $y\in E$.
Hence, we have
\begin{align*}
||\psi x||&\leq||(\psi-Q\phi_\lambda)x||+||Q\phi_\lambda x||=||(\psi-Q\phi_\lambda)x||+||\phi_\lambda(\mathbf{1})x||\\
&\leq ||(\psi-Q\phi_\lambda)x||+C||\phi_\lambda(\mathbf{1})||_1.
\end{align*}
By assumption, $||(\psi-Q\phi_\lambda)x||\to 0$ and $||\phi_\lambda(\mathbf{1})||_1\to 0$, which yields $\psi x=0$. Consequently, $\psi=0$ and $Q$ is continuous.

Finally, we should to show that $E$ is isomorphic as a topological ${}^*$- algebra to $\cM_E$ -- a complemented ${}^*$-subalgebra of $\cL^*(E,||\cdot||)$. 
Let us consider the map $\Phi\colon E\to\cM_E$, $\Phi x:=m_x$. By the above, it is clear that $\Phi$ is a ${}^*$-algebra
isomorphism. Let $B$ be a bounded subset of $E$ and let $||\cdot||_0$ be a continuous norm on $E$. 
Since the multiplication on $E$ is jointly continuous, there is $C_1>0$ and a 
continuous norm $||\cdot||_1$  on $E$ such that $||xy||_0\leq C_1||x||_1||y||_1$ for $x,y\in E$. Moreover, by the continuity of the involution, there is a constant $C_2\geq1$ and a continuous norm 
$||\cdot||_2$ on $E$ such that $||\cdot||_2\geq||\cdot||_1$ and $||x^*||_1\leq C_2||x||_2$ for $x\in E$. Hence,
\begin{align*}
\max\{\sup_{y\in B}||m_xy||_0,\sup_{y\in B}||(m_x)^*y||_0\}&=\max\{\sup_{y\in B}||xy||_0,\sup_{y\in B}||x^*y||_0\}\\
&\leq C_1\sup_{y\in B}||y||_1\max\{||x||_1,||x^*||_1\}\leq C_1C_2C_3 ||x||_2,
\end{align*}
where $C_3:=\sup_{y\in B}||y||_1<\infty$. This shows that $\Phi$ is continuous. Since $E$ and $\cM_E$ (as a complemented subspace of $\cL^*(E,||\cdot||)$, see also Corollary \ref{cor_lstarE_nuclear_ultra_PLS}) 
satisfy assumptions of the open mapping theorem \cite[Th. 24.30]{MeV}, the map $\Phi$ is an isomorphism of Fr\'echet ${}^*$-algebras, which completes the proof. 
\end{proof}

\proofmainA
The implication (i)$\Rightarrow$(ii) is trivial.

(ii)$\Rightarrow$(iii): Let $F$ be a closed ${}^*$-subalgebra of $\lstars$ such that $E\cong F$ as Fr\'echet ${}^*$-algebras and let $T\colon E\to F$ be the corresponding isomorphism.
By Lemma \ref{lem_frechet_subspaces_lstars}, there is $m\in\N_0$ such that $F\subset\cL(s_m,\ell_2)\cap\cL(\ell_2,s_{-m})$ and thus,
by Lemma \ref{lem_commutative_subalgebras_lstars}, $[\;\cdot\;]_m=\sqrt{[\cdot,\cdot]_m}$ is a dominating Hilbert norm on $F$ such that
$[xy,z]_m=[y,x^*z]_m$ for all $x,y,z\in F$.
Let $(\cdot,\cdot)\colon E\times E\to\C$, $(x,y):=[Tx,Ty]_m$. Then, clearly, $(xy,z)=(y,x^*z)$ for all $x,y,z\in E$
and $||\cdot||:=\sqrt{(\cdot,\cdot)}$ is a dominating Hilbert norm on $E$, hence $E$ is a DN-algebra.

(iii)$\Rightarrow$(i):  
By Lemma \ref{lem-M_E}, $E$ is isomorphic to a complemented ${}^*$-subalgebra of $\cL^*(E,||\cdot||)$ and, by Corollary \ref{cor_lstarF_subset_lstars}, 
$\cL^*(E,||\cdot||)$ is isomorphic to a complemented ${}^*$-subalgebra of $\lstars$, which proves the theorem. 
\bqed

\proofmainB
Clearly, (i) implies (ii).

(ii)$\Rightarrow$(iii): Let $\Phi\colon E\to F$ be the isomorphism of the Fr\'echet ${}^*$-algebras $E$ and $F$. Since $F\subset \cL(\ell_2)$, by Lemma \ref{lem_commutative_subalgebras_lstars}, 
$[\;\cdot\;]_0$ is a Hilbert dominating norm on $F$ satisfying condition ($\alpha$). 
Consequently, $||\cdot||:=[\Phi(\cdot)]_0$ is a Hilbert dominating norm  on $E$ and it satisfies condition ($\alpha$).

Next, for all $x\in E$ there is $C>0$ such that 
\[||(\Phi(xy))e_j||_{\ell_2}=||(\Phi x)((\Phi y)e_j)||_{\ell_2}\leq C||(\Phi y)e_j||_{\ell_2}\]
for all $y\in E$ and $j\in\N$. Hence, for all $x\in E$ there is $C>0$ such that
\begin{align*}
||xy||&=[\Phi(xy)]_0=\bigg(\sum_{j=1}^\infty||(\Phi(xy))e_j||^2_{\ell_2}j^{-2}\bigg)^{1/2}\leq C\bigg(\sum_{j=1}^\infty||(\Phi y)e_j||^2_{\ell_2}j^{-2}\bigg)^{1/2}\\
&=C[\Phi y]_0=C||y|| 
\end{align*}
for all $y\in E$, which gives condition ($\beta$) in Definition \ref{def-Hilbert-algebra}. 
This shows that $E$ is a $\beta\mathrm{DN}$-algebra.

(iii)$\Rightarrow$(i): Let $H$ be the completion of $E$ in the norm $||\cdot||$ and let $\cM_E:=\{m_x\colon x\in E \}$, where $m_x\colon E\to E$, $m_x(y):=xy$.
By Lemma \ref{lem-M_E}, $\cM_E$ is a complemented ${}^*$-subalgebra of $\cL^*(E,||\cdot||)$ isomorphic to $E$.
Moreover, by Lemma \ref{lemma_lstarE}, $\cL^*(E,||\cdot||)$ is isomorphic to a complemented ${}^*$-subalgebra of $\lstars$ via the map
\[\Phi\colon\cL^*(E,||\cdot||)\to\lstars,\quad x\mapsto \phi x\phi^*,\]
where $\phi\in\cL(H,\ell_2)$ and its adjoint $\phi^*\in\cL(\ell_2,H)$ satisfy conditions (i)--(iv) in Lemma \ref{lemma_lstarE}. 
Hence, the assigment $x\mapsto \phi m_x\phi^* $ defines an isomorphism (of Fr\'echet ${}^*$-algebras) between $E$ 
and a complemented ${}^*$-subalgebra of $\lstars$. Now, it is left to show that $\phi m_x\phi^*\in\cL(\ell_2)$ for each $x\in E$.
By assumption, for each $x\in E$, the map $m_x\colon (E,||\cdot||)\to(E,||\cdot||)$ is continuous. Consequently, for each $x\in E$ there are $C_1,C_2>0$ such that 
\[||\phi m_x\phi^*\xi||_{\ell_2}=||m_x\phi^*\xi||\leq C_1||\phi^*\xi||\leq C_2||\xi||_{\ell_2}\]
for all $\xi\in\ell_2$, and the proof is complete.
\bqed

\begin{remark}\label{rem-noncommutative} 
The proofs of implication (iii)$\Rightarrow$(i) in Theorems \ref{th_complemented_subalg_lstars} and \ref{th-complemented-bounded} work also when $E$ is not commutative. 
The implication (ii)$\Rightarrow$(iii) in Theorems \ref{th_complemented_subalg_lstars} and \ref{th-complemented-bounded} holds for any
commutative Fr\'echet ${}^*$-algebra.
\end{remark}

\section{Examples}\label{sec-examples}

In this section we present several examples of classes of commutative Fr\'echet ${}^*$-algebras which can be embedded into 
$\lstars$ as complemented ${}^*$-subalgebras consisting of bounded operators on $\ell_2$.
In the case of unital algebras, in view of Theorem \ref{th-complemented-bounded}, it is enough to show that a given Fr\'echet ${}^*$-algebra is isomorphic to a complemented subspace of 
$s$ with Schauder basis (i.e. to a nuclear power series space of infinite type) and admits a Hilbert dominating norm satisfying conditions ($\alpha$) and ($\beta$) in Definition \ref{def-Hilbert-algebra}.
Nonunital algebras will be extended in a natural way to unital ones .
At the end of this section we also give one interesting counterexample.

By $\cE(K)$ we denote the space of (complex-valued) Whitney jets on a compact set $K\subset\R^n$,  
\[\cE(K):=\big\{({\partial^{\alpha}F}_{\mid{K}})_{\alpha\in\N_0^n}: F\in C^\infty(\R^n)\big\}.\]
The space $\cE(K)$ thus consists of some special sequences  $f=\big(f^{(\alpha)}\big)_{\alpha\in\N_0^n}$ of continous functions on the set $K$.
The Fr\'echet space topology on $\cE(K)$ is given by the system of seminorms $(||\cdot||_m)_{m\in\N}$ defined, for example, in \cite[Section 2]{Fr07};
here let us only note that 
\[\sup\{|f^{(\alpha)}(x)|:\;x\in K,|\alpha|\leq m\}\leq||f||_m\] 
for all $m\in\N_0$ and $f\in \cE(K)$. The space $\cE(K)$ is a Fr\'echet ${}^*$-algebra where the product $fg$ of $f,g\in\cE(K)$ is defined 
by the Leibniz rule, i.e.,
\[(fg)^{(\alpha)}:=\sum_{\beta\leq\alpha}\binom{\alpha}{\beta}f^{(\beta)}g^{(\alpha-\beta)}\]
for $\alpha\in\N_0^n$ (see also \cite[p. 133]{Fr07}). As involution we clearly take the pointwise conjugation, $\overline{f}:=\big(\overline{f^{(\alpha)}}\big)_{\alpha\in\N_0^n}$. 
We say that a compact set $K\subset\R^n$ has \emph{the extension property} if there exists a continuous linear operator 
$E:\cE(K)\to C^\infty(\R^n)$ such that $\partial^\alpha(Ef)_{\mid{K}}=f^{(\alpha)}$ for every $\alpha\in\N_0^n$. 
M. Tidten showed in \cite[Folgerung 2.4]{Tid} that a compact set $K\subset\R^n$ has the extension
property if and only if $\cE(K)$ has the property (DN).

All ${}^*$-algebras of smooth functions considered below are endowed with pointwise multiplication and conjugation. The algebra $H(\C)$ of entire functions is endowed with pointwise multiplication
and the involution $f\mapsto f^*$ defined by $f^*(z):=\overline{f(\overline{z})}$.

Let  
\[\cK_\infty:=\{(x_{ij})_{i,j\in\N}\in\C^{\N^2}\colon\sup_{i,j\in\N}|x_{ij}|(ij)^n<\infty\quad\textrm{for all }n\in\N_0\}\]
be the \emph{algebra of rapidly decreasing matrices} endowed with matrix multiplication and conjugation of the transpose as involution.
The algebra $\cK_\infty$ is isomorphic as the Fr\'echet ${}^*$-algebra to the algebra $\ldss$ of compact smooth operators. Moreover, it is isomorphic as a Fr\'echet space to the space $s$. 
For further information concerning the algebra $\cK_\infty$ we refer the reader to \cite{Cias13, Cias16, Cias17}.

\begin{theorem}\label{th-examples}
The following Fr\'echet ${}^*$-algebras are isomorphic to some complemented ${}^*$-subalgebra of $\lstars$ consisting of bounded operators on $\ell_2$.
\begin{enumerate}
 \item[\upshape{(i)}] The algebras $C^\infty(M)$ of smooth functions on compact second-countable smooth manifolds $M$ without boundary.
 \item[\upshape{(ii)}] The algebras $\cD(K)$ of smooth functions on $\R^n$ with support contained in compact sets $K\subset\R^n$ such that $\operatorname{int}K\neq\emptyset$.
 \item[\upshape{(iii)}] The algebra $\cS(\R^n)$ of smooth rapidly decreasing functions on $\R^n$.
 \item[\upshape{(iv)}] The algebras $\cE(K)$ of Whitney jets on compact sets $K\subset\R^n$ with the extension property admitting a Schauder basis.
 \item[\upshape{(v)}] The algebra $H(\C)$ of entire functions.
 \item[\upshape{(vi)}] Nuclear power series algebras $\Lambda_{\infty}(\alpha)$ of infinite type with pointwise multiplication and conjugation.
 \item[\upshape{(vii)}] The algebra $\cK_\infty$ of rapidly decreasing matrices.
\end{enumerate}
\end{theorem}
\begin{proof}

(i) We follow the general pattern of the reasoning of P. Michor \cite{Mich14}. Let us choose a Riemannian metric $g$ on $M$ and let $\mathrm{d}V$ be the volume element (density) 
on $(M,g)$. Let $\langle\cdot,\cdot\rangle_{L_2(M)}$ be the scalar product on $L_2(M)$ defined by
\[\langle f,g\rangle_{L_2(M)}:=\int_M f\overline{g}\mathrm{d}V\] and let $||\cdot||_{L_2(M)}$ denote the corresponding Hilbert norm on $L_2(M)$. 
By the Sturm-Liouville decomposition \cite[pp. 139--140]{Chav84}, eigenfunctions $(u_k)_{k\in\N}$ of the Laplacian $\Delta$ induced by $g$
form an orthonormal basis of $L_2(M)$ and the sequence $(\lambda_k)_{k\in\N}$ of eigenvalues satisfies
\[0=\lambda_1<\lambda_2\leq\lambda_3\leq\ldots \text{ and }\lambda_k\to\infty.\]
Moreover, by the Weyl asymptotic formula \cite[Note III.15, p. 184]{Chav06}, there is a constant $C>0$ depending only on $n$ and the choice of a Riemannian metric such that
\begin{equation}\label{eq-Weyl}
\lambda_k\sim  Ck^{\frac{2}{n}}, 
\end{equation}
as $k\to\infty$. 

We claim that $(u_k)_{k\in\N}$ is a Schauder basis of $C^\infty(M)$ whose coefficient space is equal to the space $s$ of rapidly decreasing sequences. 
Indeed, for each $r\in\N$, the operator 
\[(I+\Delta)^r\colon H^{2r}(M)\to L_2(M)\]
is an isomorphism between the Sobolev space $H^{2r}(M)$ and $L_2(M)$. Therefore, since
\[(I+\Delta)^ru_k=(1+\lambda_k)^ru_k,\]
we have
\[u_k=(1+\lambda_k)^r(I+\Delta)^{-r}u_k\in H^{2r}(M)\]
for all $k,r\in\N$. Consequently each $u_k$ belongs to $\bigcap_{r\in\N}H^{2r}(M)=C^\infty(M)$. Next, since $(u_k)_{k\in\N}$ is an orthonormal basis of $L_2(M)$,
for each $f\in C^\infty(M)$ one can find a unique sequence $(a_k)_{k\in\N}$ of scalars such that $f=\sum_{k=1}a_ku_k$ and the series converges in the norm
$||\cdot||_{L_2(M)}$. In particular, for each $r\in\N$ there is a unique sequence $(a_{k,r})_{k\in\N}\subset\C^\N$ such that
\[(I+\Delta)^rf=\sum_{k=1}a_{k,r}u_k.\]
Since $(I+\Delta)^r$ is a symmetric unbounded operator on $L_2(M)$, we have
\[a_{k,r}=\langle(I+\Delta)^rf,u_k\rangle_{L_2(M)}=\langle f,(I+\Delta)^ru_k\rangle_{L_2(M)}=\langle f,(1+\lambda_k)^ru_k\rangle_{L_2(M)}=a_k(1+\lambda_k)^r,\]
whence
\[(I+\Delta)^rf=\sum_{k=1}a_k(1+\lambda_k)^ru_k\]
and the series converges in the norm $||\cdot||_{L_2(M)}$. Therefore, $(a_k(1+\lambda_k)^r)_{k\in\N}\in\ell_2$ for all $r\in\N$
and, by the Weyl asymptotic formula (\ref{eq-Weyl}), $(a_k(1+ Ck^{\frac{2}{n}})^r)_{k\in\N}\in\ell_2$ for all $r\in\N$, which yields $(a_k)_{k\in\N}\in s$.
Finally, it is a simple matter to show that for each $(a_k)_{k\in\N}\in s$ the series $\sum_{k=1}a_ku_k$ converges in $C^\infty(M)$.
Hence, $(u_k)_{k\in\N}$ is a Schauder basis of $C^\infty(M)$ with the coefficient space equal to $s$ as claimed. 

Now, $T\colon C^\infty(M)\to s$ defined by $Tu_k=e_k$ ($e_k$ denotes the $k$-th unit vector in $\C^\N$) is an isomorphism of Fr\'echet spaces such that 
$||Tf||_{\ell_2}=||f||_{L^2(M)}$ for $f\in C^\infty(M)$. Therefore, since $||\cdot||_{\ell_2}$ is a dominating norm on $s$, $||\cdot||_{L^2(M)}$ is a dominating norm on $C^\infty(M)$. 
Clearly, 
\[\langle fg,h\rangle_{L_2(M)}=\langle g,\overline{f}h\rangle_{L_2(M)}\quad\text{and}\quad ||fg||_{L_2(M)}\leq \sup_{x\in M}|f(x)|\cdot||g||_{L_2(M)}\]
for all $f,g,h\in C^\infty(M)$.
Hence, by Theorem \ref{th-complemented-bounded}, $C^\infty(M)$ is isomorphic to a complemented ${}^*$-subalgebra of $\lstars$ consisting of bounded operators on $\ell_2$.

(ii) Choose $a\in\R^n$ and $r>0$ such that $K$ is contained in the open ball $B(a,r)$ of radius $r$ centered at $a$.
Then, clearly, $\cD(K)$ is isomorphic as a Fr\'echet ${}^*$-algebra to the algebra  
\[\cD(K,B(a,r)):=\{f\in C^\infty(B(a,r)):\; \operatorname{supp}(f)\subset K\}\]
which is a closed subspace of the Fr\'echet space $C^\infty(\overline{B(a,r)})$ of smooth functions on $B(a,r)$ with uniformly continuous partial derivatives (see \cite[Ex. 28.9(5)]{MeV}). 
Let $\mathbf{1}$ be the constant function on $B(a,r)$, everywhere equals 1. Let $\cD_1(K,B(a,r))$ denote the linear span of $\cD(K,B(a,r))$ and $\mathbf{1}$. 
Then, clearly, $\mathbf{1}$ is the unit in the Fr\'echet ${}^*$-algebra $\cD_1(K,B(a,r))$. 
By \cite[Prop. 31.12]{MeV}, $\cD(K,B(a,r))$ is isomorphic as a Fr\'echet space to $s$, and so is $\cD_1(K,B(a,r))$. 

Now, we follow the proof of \cite[Lemma 31.10]{MeV}.
By \cite[Prop. 14.27]{MeV}, $(||\cdot||_m)_{m\in\N}$,  
\[||f||_m^2:=\sum_{|\alpha|\leq m}\int_{B(a,r)}|f^{(\alpha)}(x)|^2\mathrm{d}x=\sum_{|\alpha|\leq m}||f^{(\alpha)}||_{L_2(B(a,r))}^2,\]
is a fundamental sequence of norms on $C^\infty(\overline{B(a,r)})$, and thus on $\cD_1(K,B(a,r))$. Since 
\[(f+\lambda\mathbf{1})^{(\alpha)}=f^{(\alpha)}\]
for $|\alpha|>0$,
we have, by integration by parts and by the Cauchy-Schwartz inequality,
\begin{align*}
||(f+\lambda\mathbf{1})^{(\alpha)}||_{L_2(B(a,r))}^2&=\int_{B(a,r)}|(f+\lambda\mathbf{1})^{(\alpha)}|^2\mathrm{d}x
=\int_{B(a,r)}(f+\lambda\mathbf{1})^{(\alpha)}(\overline{f+\lambda\mathbf{1}})^{(\alpha)}\mathrm{d}x\\
&=(-1)^{|\alpha|}\int_{B(a,r)}(f+\lambda\mathbf{1})(\overline{f+\lambda\mathbf{1}})^{(2\alpha)}\mathrm{d}x\\
&\leq||f+\lambda\mathbf{1}||_{L_2(B(a,r))}||(f+\lambda\mathbf{1})^{(2\alpha)}||_{L_2(B(a,r))}
\end{align*}
for all $f+\lambda\mathbf{1}\in\cD_1(K,B(a,r))$ and $|\alpha|\leq m$.
Hence, for all $m\in\N_0$ there is a constant $C_m>0$ such that
\[||f+\lambda\mathbf{1}||^2_m\leq C_m||f+\lambda\mathbf{1}||_{L_2(B(a,r))}||f+\lambda\mathbf{1}||_{2m}\]
for all $f+\lambda\mathbf{1}\in\cD_1(K,B(a,r))$. Hence, $||\cdot||_{L_2(B(a,r))}$ is a dominating norm on $\cD_1(K,B(a,r))$, and the corresponding scalar product satisfies condition ($\alpha$) 
in Definition \ref{def-Hilbert-algebra}.
Futhermore,
\begin{align*}
||(f+\lambda\mathbf{1})(g+\mu\mathbf{1})||_{L_2(B(a,r))}&=\bigg(\int_{B(a,r)}|f(x)+\lambda|^2\,|g(x)+\mu|^2\mathrm{d}x\bigg)^{1/2}\\
&\leq\sup_{x\in B(a,r)}|f(x)+\lambda|\cdot||g+\mu\mathbf{1}||_{L_2(B(a,r))}. 
\end{align*}
Hence, $\cD_1(K,B(a,r))$ is a unital $\beta\mathrm{DN}$-algebra. 
By Theorem \ref{th-complemented-bounded}, $\cD_1(K,B(a,r))$ is isomorphic to a complemented ${}^*$-subalgebra of $\lstars$ consisting of bounded operators on $\ell_2$,
and so is $\cD(K)$.

(iii) It is well-known that the map 
\[\Phi\colon\cS(\R^n)\to\cD\bigg(\bigg[-\frac{\pi}{2},\frac{\pi}{2}\bigg]\bigg),\quad(\Phi f)(x_1,\ldots,x_n):=(\tan x_1,\ldots,\tan x_2),\]
is an isomorphism of Fr\'echet ${}^*$-algebras (see \cite[Ex. 29.5(3) and p. 402]{MeV}), hence the conclusion follows from the previous example.
Moreover, $||\cdot||\colon\cS(\R^n)\oplus\C\mathbf{1}\to[0,\infty)$, 
\[||f+\lambda\mathbf{1}||^2:=\int_{-\frac{\pi}{2}}^{\frac{\pi}{2}}\ldots\int_{-\frac{\pi}{2}}^{\frac{\pi}{2}}|f(\tan x_1,\ldots,\tan x_2)+\lambda|^2\mathrm{d}x_1\ldots\mathrm{d}x_n,\]
is a Hilbert dominating norm on $\cS(\R^n)\oplus\C\mathbf{1}$ satisfying conditions ($\alpha$) and ($\beta$).  

(iv) We shall show that there is a finite positive Borel measure $\mu$ on $K$ such that $||\cdot||_{L_2(\mu)}$,
$||f||_{L_2(\mu)}^2:=\int_{K}|f|^2\mathrm{d}\mu$,
is a dominating norm on $\cE(K)$. Then, since $||\cdot||_{L_2(\mu)}$ is a Hilbert algebra norm, we would get our conclusion.

By \cite[Th. 3.10]{Fr07}, the norm $|\cdot|_0$, $|f|_0:=\sup_{x\in K}|f(x)|$, is a dominating norm on $\cE(K)$.
This means that for all $k$ there is $l$ and $C>0$ such that
\begin{equation}\label{eq-DN0}
||f||_k\leq C||f||_l^{1/2}|f|_0^{1/2} 
\end{equation}
for all $f\in\cE(K)$.
If there were a finite positive Borel measure $\mu$ on $K$, $q\in\N$ and $C>0$ such that
\begin{equation}\label{eq-DNL2mu}
|f|_0\leq C||f||_q^{1/2}||f||_{L_2(\mu)}^{1/2} 
\end{equation}
for all $f\in\cE(K)$ then, by (\ref{eq-DN0}), we would get 
\[||f||_k\leq C||f||_l^{1/2}||f||_q^{\frac{1}{4}}||f||_{L_2(\mu)}^{\frac{1}{4}}
\leq C||f||_{\max\{l,q\}}^{\frac{3}{4}}||f||_{L_2(\mu)}^{\frac{1}{4}},\]
and hence $||\cdot||_{L_2(\mu)}$ would be a dominating norm on $\cE(K)$ (see \cite[Remark 3.2(ii)]{Fr07}). 
Our goal is thus to prove condition (\ref{eq-DNL2mu}).

In what follows, $C\geq1$ denotes a constant which can vary from line to line but depends only on $n$ and the set $K$. 
By the last line of the proof of \cite[Prop. 3.4]{BloomLev13}, 
there is a positive Borel measure $\mu$ on $K$ -- the so called Bernstein-Markov measure -- such that 
\begin{equation}\label{eq-BMproperty}
|p|_0\leq C\binom{n+j}{j}j^2||p||_{L_2(\mu)}\leq Cj^m||p||_{L_2(\mu)}.
\end{equation}
for all polynomials $p\in\C[x_1,\ldots,x_n]$ with $\deg(p)\leq j$ and where $m:=n+2$.

Let us fix $f\in\cE(K)$. By \cite[Cor. 4.4(i)]{Fr07}, for all $j\in\N$ there is a polynomial $p_j$ with $\deg(p_j)\leq j$ such that
\[|f-p_j|_0\leq Cj^{-2m}||f||_{2m}.\] 
Applying this inequality twice and also inequality (\ref{eq-BMproperty}), we obtain
\begin{align*}\label{eq-f0leq}
|f|_0&\leq|f-p_j|_0+|p_j|_0\leq C(j^{-2m}||f||_{2m}+|p_j|_0)\leq C(j^{-2m}||f||_{2m}+j^m||p_j||_{L_2(\mu)})\\
&\leq C[j^{-2m}||f||_{2m}+j^m(|f-p_j|_0+||f||_{L_2(\mu)})]\leq C(j^{-2m}||f||_{2m}+j^{-m}||f||_{2m}+j^m||f||_{L_2(\mu)})\\
&\leq C(j^{-m}||f||_{2m}+j^m||f||_{L_2(\mu)})
\end{align*}
for all $j\in\N$. Therefore, taking the infimum over $j\in\N$ in both sides of the above chain of inequalities and applying \cite[Lemma 4.5]{Fr07}, we obtain
\[|f|_0\leq C||f||_{2m}^{1/2}||f||_{L_2(\mu)}^{1/2}\]
for all $f\in\cE(K)$, and (\ref{eq-DNL2mu}) holds, as desired.

(v) We will show that the norm $||\cdot||_{L_2[-1,1]}$, which obviously satisfies conditions ($\alpha$) and ($\beta$) in Definition \ref{def-Hilbert-algebra}, is at the same time a dominating norm.
Let us recall that the Fr\'echet space topology on $H(\C)$ is given by the sequence of norms $(||\cdot||_r)_{r\in\N}$, $||f||_r:=\sup_{|z|=r}{|f(z)|}$. 
By Hadamard's three circle theorem, 
\begin{equation}\label{eq_hadamard}
||f||_r^2\leq||f||_1\;||f||_{r^2} 
\end{equation}
for all $r>1$ and all $f\in H(\C\setminus\{0\})$, where $H(\C\setminus\{0\})$ is 
the space of holomorphic functions on $\C\setminus\{0\}$.
Let us consider the map $\Psi\colon\C\setminus\{0\}\to\C$, $\Psi(z):=\frac{1}{2}(z+z^{-1})$ and let $E_r:=\Psi(\mathbb{T}_r)$, 
where $\mathbb{T}_r:=\{z\in\C:|z|=r\}$ for $r>0$. 
For $r>1$, one can show that $E_r$ is the ellipse with the semi-axes $a=\frac{1}{2}(r+r^{-1})$ and $b=\frac{1}{2}(r-r^{-1})$. Moreover, $E_1=[-1,1]$ and $E_{r^{-1}}=E_r$ for $r>0$.
Since $f\circ\Psi\in H(\C\setminus\{0\})$ for all $f\in H(\C)$, we have, by (\ref{eq_hadamard}), 
\[||f\circ\Psi||_r^2\leq||f\circ\Psi||_1\;||f\circ\Psi||_{r^2}.\]
Hence,
\[||f||_{E_r}^2\leq||f||_{[-1,1]}\;||f||_{E_{r^2}}\]
for all $r>1$ and all $f\in H(\C)$, where $||f||_E$ is the supremum norm on a subset $E$ of $\C$. Since the ellipses $E_r$ contain arbitrary big circles, $||\cdot||_{[-1,1]}$ is a dominating norm on $H(\C)$.

Let $(Q_k)_{k\in\N_0}$ be the sequence of Legendre polynomials,
\[Q_k(x):=\bigg(\frac{2k+1}{2}\bigg)^{1/2}\frac{1}{2^kk!}\frac{\mathrm{d}^k}{\mathrm{d}x^k}(x^2-1)^k\]
for $x\in[-1,1]$. Then, for each $n\in\N$, $(Q_k)_{0\leq k\leq n}$ is an orthonormal basis of the space $\cP_n$ of complex-valued polynomials of 
degree at most $n$ in one real variable endowed with the scalar product $\langle p,q \rangle:=\int_{-1}^1 p(x)\overline{q(x)}\mathrm{d}x$. 
It is well-known that 
\[|Q_k(x)|\leq\bigg(\frac{2k+1}{2}\bigg)^{1/2}\]
for all $x\in[-1,1]$. Therefore, if $p\in\cP_n$, $p=\sum_{k=0}^n c_kQ_k$ for some scalars $c_k$'s, then 
\[|p(x)|\leq\left(\sum_{k=0}^n|Q_k(x)|^2\right)^{1/2}\left(\sum_{k=0}^n |c_j|^2\right)^{1/2}\leq(n+1)||p||_{L_2[-1,1]}\]
for every $x\in[-1,1]$. Consequently,
\begin{equation}\label{eq_HC1}
||p||_{[-1,1]}\leq(n+1)||p||_{L_2[-1,1]} 
\end{equation}
for every $p\in\cP_n$. This shows that the Lebesgue measure is a Bernstein-Markov measure on the interval $[-1,1]$.

Now, let us take $f\in H(\C)$, $f=\sum_{j=0}^\infty a_jz^j$. Let $p_n:=\sum_{j=0}^n a_jz^j$ for $n\in\N_0$. Then, for every $n\in\N_0$ and $x\in[-1,1]$, 
we have
\begin{align*}
|f(x)-p_n(x)|&=\bigg|\sum_{j=n+1}^\infty a_jx^j\bigg|=\bigg|\sum_{j=n+1}^\infty a_jx^{j-n-1}\bigg|\cdot|x|^{n+1}\\
&\leq\frac{1}{(n+1)!}\bigg|\sum_{j=0}^\infty (j+n+1)\ldots(j+1)a_{j+n+1}x^j\bigg|
=\frac{1}{(n+1)!}|f^{(n+1)}(x)|\\
&=\frac{1}{2\pi}\bigg|\int_{|z-x|=2}\frac{f(z)}{(z-x)^{n+2}}\mathrm{d}z\bigg|
\leq\frac{1}{2^{n+1}}\sup_{|z-x|=2}|f(z)|\leq\frac{1}{2^{n+1}}\sup_{|z|\leq3}|f(z)|.
\end{align*}
Hence,
\begin{equation}\label{eq_HC2}
||f-p_n||_{[-1,1]}\leq\frac{1}{2^{n+1}}||f||_3. 
\end{equation}

Now, by applying computations from item (iv), we may derive from inequalities (\ref{eq_HC1}) and (\ref{eq_HC2}) that $||\cdot||_{L_2[-1,1]}$ is a dominating norm
on $H(\C)$ as claimed.

(vi) Let us define
\[||\cdot||\colon\Lambda_{\infty}(\alpha)\oplus\C\mathbf{1}\to[0,\infty),\quad||x+\lambda||^2:=\sum_{j=1}^\infty|x_j+\lambda|^2j^{-2}\]
and
\[||\cdot||_t\colon\Lambda_{\infty}(\alpha)\oplus\C\mathbf{1}\to[0,\infty),\quad||x+\lambda||_t:=\max\{\sup_{j\in\N}|x_j|e^{t\alpha_j},|\lambda|\}\]
for $t\in\N$, where $\mathbf{1}:=(1,1,\ldots)$.
By nuclearity, $(||\cdot||_t)_{t\in\N}$ is a fundamental sequence of norms on $\Lambda_{\infty}(\alpha)\oplus\C\mathbf{1}$
and $||\cdot||$ is a well-defined continuous Hilbert norm on $\Lambda_{\infty}(\alpha)\oplus\C\mathbf{1}$.
Let $(\cdot,\cdot)$ denote the scalar product corresponding to the norm $||\cdot||$. An elementary computation shows that 
\[((x+\lambda)(y+\mu),z+\nu)=((y+\mu),(\overline{x+\lambda})(z+\nu))\]
for all $x+\lambda$, $y+\mu$, $z+\nu\in\Lambda_{\infty}(\alpha)\oplus\C\mathbf{1}$.
Moreover, we have
\[||(x+\lambda)(y+\mu)||=\bigg(\sum_{j=1}^\infty|(x_j+\lambda)(y_j+\mu)|^2j^{-2}\bigg)^{1/2}\leq\sup_{j\in\N}|x_j+\lambda|\,||y+\mu||\]
for all $x+\lambda$, $y+\mu\in\Lambda_{\infty}(\alpha)\oplus\C\mathbf{1}$. Hence, the norm $||\cdot||$ satisfies conditions ($\alpha$) and ($\beta$) in Definition \ref{def-Hilbert-algebra}.  

We will show that  $||\cdot||$ is a dominating norm on $\Lambda_{\infty}(\alpha)\oplus\C\mathbf{1}$, i.e. 
\begin{equation}\label{eq-DN-Lambda+1}
\forall s\in\N\;\exists t\in\N\;\exists C>0\;\forall x+\lambda\in\Lambda_{\infty}(\alpha)\oplus\C\mathbf{1}\quad ||x+\lambda||_s^2\leq C||x+\lambda||\cdot||x+\lambda||_t.
\end{equation}

Fix $s\in\N$, $x+\lambda\in\Lambda_{\infty}(\alpha)\oplus\C\mathbf{1}$ and denote 
\[R(t):=||x+\lambda||\cdot||x+\lambda||_t\] 
for $t\in\N$. 
Next, note that
\[\frac{\log(j+1)}{\alpha_j}\leq\frac{\log(ej^2)}{\alpha_j}=\frac{1}{\alpha_j}+2\frac{\log j}{\alpha_j}\]
for every $j\in\N$
and, by the nuclearity of $\Lambda_{\infty}(\alpha)$, there is a constant $\gamma\in\N$ such $\frac{1}{\alpha_j}+2\frac{\log j}{\alpha_j}\leq\gamma$ for all $j\in\N$.
Hence, $\frac{\log(j+1)}{\alpha_j}\leq\gamma$, and thus
\[\frac{e^{\gamma\alpha_j}}{j+1}\geq1\]
for all $j\in\N$.
We first claim that 
\begin{equation}\label{eq-lambda2}
|\lambda|^2\leq 4R(\gamma). 
\end{equation}
Indeed, since $|x_j+\lambda|\to|\lambda|$ as $j\to\infty$, the set $A:=\{j\in\N:\;|x_j+\lambda|<\frac{|\lambda|}{2}\}$ is finite or even empty. If $A=\emptyset$, i.e. $|x_j+\lambda|\geq\frac{|\lambda|}{2}$ for all $j\in\N$,
then 
\[R(\gamma)\geq|x_1+\lambda|\frac{|\lambda|}{2}\geq\frac{|\lambda|^2}{2}.\]

Now, assume that $A\neq\emptyset$ and let $j_0:=\max\{j\in\N:\;|x_j+\lambda|<\frac{|\lambda|}{2}\}$. Then, clearly, $|x_{j_0+1}+\lambda|\geq\frac{|\lambda|}{2}$, which gives
\[||x+\lambda||\geq|x_{j_0+1}+\lambda|(j_0+1)^{-1}\geq\frac{|\lambda|}{2}(j_0+1)^{-1}.\]
Next, $|x_{j_0}+\lambda|<\frac{|\lambda|}{2}$ implies that $|x_{j_0}|>\frac{|\lambda|}{2}$, and thus
\[||x+\lambda||_t\geq|x_{j_0}|e^{t\alpha_{j_0}}>\frac{|\lambda|}{2}e^{t\alpha_{j_0}}\]
for all $t\in\N$. Consequently,
\[R(\gamma)\geq\frac{|\lambda|}{2}(j_0+1)^{-1}\cdot\frac{|\lambda|}{2}e^{\gamma\alpha_{j_0}}=\frac{|\lambda|^2}{4}\frac{e^{\gamma\alpha_{j_0}}}{j_0+1}\geq\frac{|\lambda|^2}{4}\]
as claimed.

We next claim that $\sup_{j\in\N}|x_j|^2e^{2s\alpha_j}\leq 16R(2s+\gamma)$. Indeed, let $t:=2s+\gamma$ and let us divide the set of natural numbers into 3 pieces:
\begin{align*}
A_1&:=\bigg\{j\in\N:\;|x_j|\leq\frac{|\lambda|}{2}\bigg\},\\ 
A_2&:=\bigg\{j\in\N:\;\frac{|\lambda|}{2}<|x_j|\leq2|\lambda|\bigg\},\\
A_3&:=\{j\in\N:\;|x_j|>2|\lambda|\}.
\end{align*}
First take $k\in A_1$. Then $|x_k+\lambda|\geq\frac{|\lambda|}{2}$, and so
\begin{align}\label{eq-DN-A1}
\begin{split}
R(t)&\geq|x_k+\lambda|k^{-1}|x_k|e^{t\alpha_k}\geq\frac{|\lambda|}{2}|x_k|\frac{e^{t\alpha_k}}{k}\geq|x_k|^2\frac{e^{t\alpha_k}}{k}\geq|x_k|^2\frac{e^{t\alpha_k}}{k+1}
=|x_k|^2e^{(t-\gamma)\alpha_k}\frac{e^{\gamma\alpha_k}}{k+1}\\
&\geq|x_k|^2e^{(t-\gamma)\alpha_k}=|x_k|^2e^{2s\alpha_k}. 
\end{split}
\end{align}

Now, take $k\in A_2$. Clearly the set $A_2$ is finite. Let $k_0:=\max\{j\in\N:\;|x_j|>\frac{\lambda}{2}\}$. Then $|x_{k_0+1}|\leq\frac{|\lambda|}{2}$, and thus $|x_{k_0+1}+\lambda|\geq\frac{|\lambda|}{2}$.
Consequently, we get
\begin{align}\label{eq-DN-A2}
\begin{split}
R(t)&\geq|x_{k_0+1}+\lambda|(k_0+1)^{-1}|x_{k_0}|e^{t\alpha_{k_0}}\geq\frac{|\lambda|}{2}(k_0+1)^{-1}\frac{|\lambda|}{2}e^{t\alpha_{k_0}}=\frac{|\lambda|^2}{4}\frac{e^{t\alpha_{k_0}}}{k_0+1}\\
&\geq\frac{|x_k|^2}{16}\frac{e^{t\alpha_{k_0}}}{k_0+1}=\frac{|x_k|^2}{16}e^{(t-\gamma)\alpha_{k_0}}\frac{e^{\gamma\alpha_{k_0}}}{k_0+1}\geq\frac{|x_k|^2}{16}e^{2s\alpha_k}.
\end{split}
\end{align}

Finally, fix $k\in A_3$. Then $|x_k+\lambda|>\frac{|x_k|}{2}$, and thus
\begin{equation}\label{eq-DN-A3}
R(t)\geq|x_k+\lambda|k^{-1}|x_k|e^{t\alpha_k}\geq\frac{|x_k|^2}{2}\frac{e^{t\alpha_k}}{k}\geq\frac{|x_k|^2}{2}e^{(t-\gamma)\alpha_k}\frac{e^{\gamma\alpha_k}}{k+1}\geq\frac{|x_k|^2}{2}e^{2s\alpha_k}.
\end{equation}
Combining (\ref{eq-DN-A1})--(\ref{eq-DN-A3}), we get
\[\sup_{j\in\N}|x_j|^2e^{2s\alpha_j}\leq16R(2s+\gamma)\]
as claimed. Therefore, by (\ref{eq-lambda2}), 
\[||x+\lambda||_s^2\leq 16||x+\lambda||\cdot||x+\lambda||_{2s+\gamma},\]
for all $x+\lambda\in\Lambda_{\infty}(\alpha)\oplus\C\mathbf{1}$, and $||\cdot||$ is a dominating norm on $\Lambda_{\infty}(\alpha)\oplus\C\mathbf{1}$.
Consequently, $(\Lambda_{\infty}(\alpha)\oplus\C\mathbf{1},(\cdot,\cdot))$ is a $\beta$DN-algebra. 

Finally, by the above, $\Lambda_{\infty}(\alpha)\oplus\C\mathbf{1}$ has the property (DN) and, by \cite[Ex. 1, Ch. 29]{MeV}, it has the property ($\Omega$). Moreover, by
\cite[Prop. 28.7]{MeV}, the space $\Lambda_{\infty}(\alpha)\oplus\C\mathbf{1}$ is nuclear. Consequently, by \cite[Prop. 31.7]{MeV}, $\Lambda_{\infty}(\alpha)\oplus\C\mathbf{1}$ 
is isomorphic to a complemented subspace of $s$ with Schauder basis, i.e. to a nuclear power series space of infinite type. Now, the thesis follows from Theorem \ref{th-complemented-bounded}
and from a simple observation that $\Lambda_{\infty}(\alpha)$ is a complemented ${}^*$-subalgebra of $\Lambda_{\infty}(\alpha)\oplus\C\mathbf{1}$.

(vii) Let $\mathbf{1}$ be the identity matrix in $\C^{\N^2}$ and let $e_j$ be the $j$-th unit vector in $\C^\N$. 
We define on $\cK_\infty\oplus\C\mathbf{1}$ the Hilbert norm $||\cdot||$ by 
\[||x+\lambda||^2:=\bigg(\sum_{j\in\N}||(x+\lambda)e_j||_{\ell_2}^2j^{-2}\bigg)^{1/2}.\]
Clearly, $||\cdot||$ satisfies conditions ($\alpha$) and ($\beta$) in Definition \ref{def-Hilbert-algebra}.
Hence, by Theorem \ref{th-complemented-bounded} and Remark \ref{rem-noncommutative}, it is enough to show that $||\cdot||$ is a dominating norm, i.e.
\[\forall k\in\N\;\exists n\in\N\;\exists C>0\;\forall x+\lambda\in\cK_\infty\oplus\C\mathbf{1}\quad ||x+\lambda||_k^2\leq C||x+\lambda||\cdot||x+\lambda||_n,\]
where 
\[||x+\lambda||_k:=\max\{\sup_{i,j\in\N}|x_{ij}|(ij)^k,|\lambda|\}.\]

Let us fix $x+\lambda\in\cK_\infty\oplus\C\mathbf{1}$. Note that
\[||(x+\lambda)e_j||_{\ell_2}^2=||(x_{ij})_{i\in\N}+\lambda e_j||_{\ell_2}^2=\sum_{i\in\N\setminus\{j\}}|x_{ij}|^2+|x_{jj}+\lambda|^2\]
for every $j\in\N$, and thus
\[||x+\lambda||^2=\sum_{j\in\N}\bigg(\sum_{i\in\N\setminus\{j\}}|x_{ij}|^2+|x_{jj}+\lambda|^2\bigg)j^{-2}.\]
We have
\begin{align*}
\sup_{i,j\in\N,i\neq j}|x_{ij}|^2(ij)^{2k}&\leq\sup_{i,j\in\N,i\neq j}|x_{ij}|j^{-1}\cdot\sup_{i,j\in\N}|x_{ij}|(ij)^{2k+1}\\
&\leq\bigg(\sum_{i,j\in\N,i\neq j}|x_{ij}|^2j^{-2}\bigg)^{1/2}\cdot\sup_{i,j\in\N}|x_{ij}|(ij)^{2k+1}\\
&\leq\bigg(\sum_{j\in\N}\bigg(\sum_{i\in\N\setminus\{j\}}|x_{ij}|^2+|x_{jj}+\lambda|^2\bigg)j^{-2}\bigg)^{1/2}\cdot\max\{\sup_{i,j\in\N}|x_{ij}|(ij)^{2k+1},|\lambda|\}\\
&=||x+\lambda||\cdot||x+\lambda||_{2k+1}
\end{align*}
for all $k\in\N$.
Moreover, from (\ref{eq-DN-Lambda+1}) applied to the algebra $s\oplus\C\mathbf{1}$, it follows that for every $k\in\N$ there is $m\in\N$ and $C>0$ 
such that  
\begin{align*}
\max\{\sup_{j\in\N}|x_{jj}|^2j^{4k},|\lambda|^2\}&\leq C\bigg(\sum_{j\in\N}|x_{jj}+\lambda|^2j^{-2}\bigg)^{1/2}\max\{\sup_{j\in\N}|x_{jj}|j^{2m},|\lambda|\}\\
&\leq C\bigg(\sum_{j\in\N}\bigg(\sum_{i\in\N\setminus\{j\}}|x_{ij}|^2+|x_{jj}+\lambda|^2\bigg)j^{-2}\bigg)^{1/2}\cdot\max\{\sup_{i,j\in\N}|x_{ij}|(ij)^m,|\lambda|\}\\
&=C||x+\lambda||\cdot||x+\lambda||_m.
\end{align*}
Therefore, for all $k\in\N$ there is $n\in\N$ and $C>0$ such that
\[||x+\lambda||_k^2=C\max\bigg\{\sup_{i,j\in\N,i\neq j}|x_{ij}|^2(ij)^{2k},\max\big\{\sup_{j\in\N}|x_{jj}|^2j^{4k},|\lambda|^2\big\}\bigg\}
\leq C||x+\lambda||\cdot||x+\lambda||_n,\]
and thus $||\cdot||$ is a dominating norm on $\cK_\infty\oplus\C\mathbf{1}$, which completes the proof.
\end{proof}

\begin{remark}
(i) Every algebra $\cD(K)$ is a closed ${}^*$-subalgebra of $\cE(L)$ for any closed ball $L$ containing $K$ and thus, by Theorem \ref{th-examples}(ii), $\cD(K)$ is automatically
isomorphic to a \emph{closed} ${}^*$-subalgebra of $\lstars$ consisting of bounded operators on $\ell_2$. In Theorem \ref{th-examples}(iv),
we prove that such a ${}^*$-subalgebra can be choosen to be \emph{complemented}.

(ii) By \cite[Proposition 4.3]{CiasPisz}, $\lstars$ is isomorphic as a topological ${}^*$-algebra to the matrix algebra
\[\Lambda(\cA):=\bigg\{x=(x_{ij})\in\C^{\N^2}:\forall N\in\N\,\exists n\in\N\quad\sum_{i,j\in\N^2}|x_{ij}|\max\bigg\{\frac{i^N}{j^n},\frac{j^N}{i^n}\bigg\}<\infty\bigg\}\]
endowed with the so-called K\"othe-type PLB-space topology. Clearly, $\cK_\infty$ is a ${}^*$-subalgebra of $\Lambda(\cA)$. However, since $\cK_\infty$ is dense in $\Lambda(\cA)$,
this representation is not interesting for us.

We have also a simillar phenomenon in the case of the algebra $s$. The diagonal matrices from $\Lambda(\cA)$ give -- also in the topological sense -- exactly the space $s'$. 
In particular, the ${}^*$-algebra $s$ has a simple representation in $\Lambda(\cA)$. But $s$ is dense in $s'$, so this representation does not work for us.

In the proof of Theorem \ref{th-examples}, Examples (vi) and (vii), we see that the representations of $s$ and $\cK_\infty$ in $\lstars$ are much more sofisticated.
\end{remark}

Finally, we shall give an example of a unital commutative Fr\'echet ${}^*$-algebra isomorphic as a Fr\'echet space to $s$ which is not a DN-algebra. Let $A^\infty(\D)$ be the space of holomorphic 
functions on the open unit disc $\D$ which are smooth up to boundary. In other words, 
\[A^\infty(\D):=\{f\in A(\D): f^{(k)}\in A(\D)\text{ for all $k\in\N$}\},\]
where $A(\D)$ is the disc algebra. The space $A^\infty(\D)$ admits a natural Fr\'echet space topology given by the norms 
\begin{equation}\label{eq-norms_A_infty}
||f||_p:=\sup\{|f^{(j)}(z)|:z\in\D,0\leq j\leq p\},\quad p\in\N_0 
\end{equation}
and it is isomorphic to the space $s$ (cf. \cite[Section 2]{{TayWill70}}).
Moreover, $A^\infty(\D)$ becomes a unital commutative Fr\'echet ${}^*$-algebra when endowed with the usual multiplication of functions and involution $f^*(z):=\overline{f(\overline{z})}$.
The algebra $A^\infty(\D)$ is thus a (dense) ${}^*$-subalgebra of the disc algebra.

\begin{proposition}\label{smooth-disc-algebra}
The algebra $A^\infty(\D)$ is not isomorphic to any closed ${}^*$-subalgebra of $\lstars$. 
\end{proposition}
\begin{proof}
By Theorem \ref{th_complemented_subalg_lstars}, it is enough to show that there is no DN-norm on $A^\infty(\D)$. 

Let $(||\cdot||_p)_{p\in\N}$ be the fundamental sequence of norms on the space $A^\infty(\D)$ defined by (\ref{eq-norms_A_infty}). First, we will show that the norm 
$||\cdot||_{[-1,1]}$, $||f||_{[-1,1]}:=\sup_{x\in[-1,1]}|f(x)|$, is not a dominating norm on $A^\infty(\D)$. Let $f_n(z):=(z^2-1)^n$ for $n\in\N$ and $z\in\D$.
Let us fix  $n\in\N$ and $0\leq k\leq n$. Then, clearly, $||f_n||_{[-1,1]}=1$  and $||f_n||_0=2^n$. Moreover,
by the Leibniz rule, we obtain
\begin{align*}
|f_n^{(k)}(z)|&=\bigg|\sum_{j=0}^k\binom{k}{j}\bigg(\frac{\mathrm{d}}{\mathrm{d}z}\bigg)^j(z-1)^n\bigg(\frac{\mathrm{d}}{\mathrm{d}z}\bigg)^{k-j}(z+1)^n\bigg|\\
&=\bigg|\sum_{j=0}^k\binom{k}{j}\cdot n(n-1)\ldots(n-j+1)(z-1)^{n-j}\cdot n(n-1)\ldots(n-k+j+1)(z+1)^{n-k+j}\bigg|\\
&\leq n^k2^k\cdot \sup_{z\in\D,0\leq j\leq k}|z-1|^{n-j}|z+1|^{n-k+j}\\
&=n^k2^{\frac{k}{2}}2^n\cdot\sup_{0\leq\theta\leq\frac{\pi}{2},0\leq j\leq k}(1-\cos{\theta})^{\frac{n-j}{2}}(1+\cos{\theta})^{\frac{n-k+j}{2}}
\leq n^k2^k2^n
\end{align*}
for all $z\in\D$, and thus
\[||f_n||_p\leq n^p2^p2^n\]
for every $p\in\N_0$.
Consequently, for every $p\in\N$ and every $C>0$ there is $n\in\N$ such that
\[||f_n||_0^2=4^n>Cn^p2^p2^n\geq C||f_n||_{[-1,1]}\;||f_n||_p,\]
and therefore $||\cdot||_{[-1,1]}$ is not a dominating norm on $A^\infty(\D)$. 

Now, let $||\cdot||=\sqrt{(\cdot,\cdot)}$ be an abitrary continuous norm on $A^\infty(\D)$ satisying condition ($\alpha$) in Definition \ref{def-Hilbert-algebra}. 
We define a continuous linear functional $\Phi$ on $A^\infty(\D)$ by $\Phi(f):=(f,\mathbf{1})$,
where $\mathbf{1}$ is the identically one function. Then, by condition ($\alpha$),
\[\Phi(f^*f)=(f^*f,\mathbf{1})=||f||^2\geq0\]
for every $f\in A^\infty(\D)$, so $\Phi$ is a positive functional. In \cite{Pav89}, there is an elementary proof of the fact that each positive linear functional on the disc algebra 
$A(\D)$ has some simple integral representation. It appears that the same proof works in the case of the algebra $A^\infty(\D)$, and thus there is a positive Borel measure $\mu$ on $[-1,1]$
such that
\[\Phi(f)=\int_{-1}^1f(x)\mathrm{d}\mu(x)\]
for every $f\in A^\infty(\D)$. Hence,
\[||f||^2=\int_{-1}^1|f(x)|^2\mathrm{d}\mu(x).\]
If $||\cdot||$ was a dominating norm on $A^\infty(\D)$ then, since $||\cdot||\leq\mu([-1,1])||\cdot||_{[-1,1]}$, the norm $||f||_{[-1,1]}$ would be a dominating norm as well, contrary to our 
first claim. Hence, $A^\infty(\D)$ is not a DN-algebra, which completes the proof.
\end{proof}

\begin{Acknow*}
The author wishes to express his thanks to Leonhard Frerick for many stimulating conversations, especially during the stay at Trier University in June 2017. 
The author is also very endebted to his colleague Krzysztof Piszczek for valuables comments during the manuscript preparation. 
\end{Acknow*}

\vspace{1cm}
\begin{minipage}{8.5cm}
Tomasz Cia\'s

Faculty of Mathematics and Computer Science

Adam Mickiewicz University in Pozna{\'n}

Umultowska 87

61-614 Pozna{\'n}, POLAND

e-mail: \url{tcias@amu.edu.pl}
\end{minipage}\

\end{document}